\documentclass[11pt,leqno]{amsart}
\usepackage{amsmath,amsfonts,amsthm,amscd,amssymb,graphicx}
\usepackage{epstopdf}
\numberwithin{equation}{section}
\usepackage{fullpage}

\def\eps{\varepsilon }

\newcommand\R{\mathbb R}
\newcommand\C{\mathbb C}

\def\eps{\varepsilon}

\newcommand\br{\begin{remark}}
\newcommand\er{\end{remark}}
\newcommand\brs{\begin{remarks}}
\newcommand\ers{\end{remarks}}
\newcommand\bp{\begin{pmatrix}}
\newcommand\ep{\end{pmatrix}}
\newcommand{\be}{\begin{equation}}
\newcommand{\ee}{\end{equation}}
\newcommand\ba{\begin{equation}\begin{aligned}}
\newcommand\ea{\end{aligned}\end{equation}}

\newcommand{\bap}{\begin{app}}
\newcommand{\eap}{\end{app}}
\newcommand{\begs}{\begin{exams}}
\newcommand{\eegs}{\end{exams}}
\newcommand{\beg}{\begin{example}}
\newcommand{\eeg}{\end{exaplem}}
\newcommand{\bpr}{\begin{proposition}}
\newcommand{\epr}{\end{proposition}}
\newcommand{\bt}{\begin{theorem}}
\newcommand{\et}{\end{theorem}}
\newcommand{\bc}{\begin{corollary}}
\newcommand{\ec}{\end{corollary}}
\newcommand{\bl}{\begin{lemma}}
\newcommand{\el}{\end{lemma}}
\newcommand{\bd}{\begin{definition}}
\newcommand{\ed}{\end{definition}}

\newcommand{\CalT}{\mathcal{T}}

\newcommand{\RR}{{\mathbb R}}

\newcommand{\CC}{{\mathbb C}}

\newcommand{\Id}{{\rm Id }}

\newcommand{\diag}{{\rm diag }}
\newcommand{\blockdiag}{{\rm blockdiag }}

\newtheorem{theorem}{Theorem}[section]
\newtheorem{proposition}[theorem]{Proposition}
\newtheorem{corollary}[theorem]{Corollary}
\newtheorem{lemma}[theorem]{Lemma}

\theoremstyle{remark}
\newtheorem{remark}[theorem]{Remark}
\newtheorem{remarks}[theorem]{Remarks}
\theoremstyle{definition}
\newtheorem{definition}[theorem]{Definition}

\newtheorem{example}[theorem]{Example}

\newcommand\cA{{\mathcal { A}}}
\newcommand\cB{{\mathcal  B}}

\newcommand\cM{{\mathcal M}}

\newcommand{\bbR}{{\mathbb{R}}}

\newcommand{\beq}{\begin{equation}}
\newcommand{\eeq}{\end{equation}}

\title{
Block-diagonalization of ODEs in the semiclassical limit and $C^\omega$ vs. $C^\infty$ stationary phase
}

\author{Olivier Lafitte}
\address{ Universit\'e de Paris 13, LAGA and CEA Saclay, DM2S}
\email{lafitte@math.univ-paris13.fr}
\author{Mark Williams}
\address{University of North Carolina, Chapel Hill}
\email{ williams@email.unc.edu}
\thanks{Research of M.W. was partially supported by
NSF grants number DMS-0701201 and DMS-1001616}
\author{Kevin Zumbrun}
\address{Indiana University, Bloomington, IN 47405}
\email{kzumbrun@indiana.edu} 
\thanks{Research of K.Z. was partially supported
under NSF grants no. DMS-0300487 and DMS-0801745.}
\begin{document}

\begin{abstract}
Motivated by issues in detonation stability, we study existence of block-diagonalizing transformations 
for ordinary differential semiclassical limit problems arising in the study of 
high-frequency eigenvalue problems.  Our main results
are to (i) establish existence of block-diagonalizing transformations in
a neighborhood of infinity for analytic-coefficient ODE, and
(ii) establish by a series of counterexample sharpness of hypotheses
and conclusions on existence of block-diagonalizing transformations near a finite point.
In particular, we show that, in general, bounded transformations exist only
locally, answering a question posed by Wasow in the 1980's,
and, under the minimal condition of spectral separation, 
for ODE with analytic rather than $C^\infty$ coefficients.
The latter issue is connected with 
quantitative comparisons of $C^\omega$ vs. $C^\infty$ stationary phase estimates.
\end{abstract}
\date{\today}
\maketitle


\section{Introduction}
Motivated by problems in detonation 
and related hydrodynamical and continuum-mechanical stability, 
we consider the general {\it semi-classical limit} problem
\be\label{semi}
h (d/dx) Z=(A(x,h;q)+h B(x,h;q) ) Z,
\quad Z\in \CC^N,
\quad h\to 0^+, 
\ee
on a possibly unbounded domain $x\in [a,b]\subset \R$,
representing a generalized spectral problem with
wavelength $h\in \R^+$ and frequency $k=1/h$.
Here, $q \in \R^s$, bounded, records any additional parameters associated with the problem:
typically, spectral angle and or bifurcation parameters.

Such systems arise for example as generalized eigenvalue problems
for the linearized equations 
\be\label{e:lin}
m_t=Lv:=\sum_\alpha a_\alpha(x_1) \partial^\alpha_x v,
\qquad 
\partial^\alpha_x:=\partial_{x_1}^{\alpha_1}\cdots \partial_{x_d}^{\alpha_d},
\ee
about a steady planar solution
\be\label{planar}
w(x_1, \dots, x_d,t)=\bar w(x_1)
\ee
of a PDE 
\be\label{pde}
w_t = \mathcal{F}(w):= \sum_{\alpha}f_\alpha(w) \partial^\alpha_x w, 
\ee
in the high-frequency limit.
Specifically, taking the Fourier transform in $(x_2,\dots, x_d)$, reduces the eigenvalue problem
$\lambda v =Lv$ to an ODE 
\be\label{fulle}
\lambda \hat v=L_\xi \hat v:= \sum_\alpha a_\alpha(x_1) \partial_{x_1}^{\alpha_1}
(i\xi_2)^{\alpha_2}\cdots (i\xi_d)^{\alpha_d}.
\ee

Writing \eqref{fulle} as a first-order system in a suitable phase variable $Z$ including $v$ and appropriate 
additional $x_1$-derivatives,
and defining $h^{-1}:=|\xi, \lambda|$ as spectral frequency and 
$$
q=(\hat \lambda, \hat \xi):=(\xi/|\lambda,\xi|,\xi/|\lambda,\xi|)
$$
as spectral angle,
we arrive in the high-frequency limit $|\xi,\lambda|\to \infty$ at a problem of form \eqref{semi} in the variable $x=x_1$.
In this context, values $(q, h)$ for which there
exist solutions $Z$ of \eqref{semi} satisfying appropriate boundary
conditions at endpoints $x=a,b$ correspond to spectra
$\lambda= \hat \lambda/h$ of the Fourier transform $L_\xi$ of the linearized operator about the wave,
hence an understanding of small-$h$ behavior of \eqref{semi}
corresponds to an understanding of high-frequency spectral stability.
See \cite{Er,Z1,LWZ1,LWZ2} for specific examples pertaining to stability of detonation waves.

Our goal in this paper is a systematic treatment of {\it local block reduction} of \eqref{semi}, 
or decomposition of the equations into spectrally separated blocks possessing
nontrivial turning points, in particular in the important case, not previously treated to our knowledge,
of a neighborhood of plus or minus infinity.
At finite points, for which existence of locally diagonalizing transformations has been exhaustively studied in \cite{W,O}, 
our goal is to determine sharpness of hypotheses and conclusions, and in particular compare results obtainable by complex-analytic methods to those obtained by $C^r$ methods in, e.g., \cite{Z1,LWZ1}.
The treatment of the resulting smaller blocks after this decomposition, and the 
global implications for stability
are studied, for example, in \cite{W,O,LWZ2,LWZ3}.

For simplicity of exposition, we will suppress in what follows dependence on 
the parameter $q$ (corresponding in example \eqref{fulle}
to restriction to the 1D case $d=1$), leaving only dependence of coefficients on the parameter $h$.
However, it is an important point that all estimates of the paper carry over to the general case, {\it uniformly in the 
parameters $q$ and $h$,}
the treatment of $q$-dependence being no different than the
treatment of dependence on $h$.
As noted in \cite{LWZ1}, such uniform estimates are important in the verification
of high-frequency {\it stability}, or nonexistence of unstable spectra for
$h$ sufficiently small, a property involving all parameter values, 
as opposed to {\it instability}, a property that need be checked only at isolated strategically chosen parameter values.

This, and the treatment for unbounded domains
of block-diagonalization at infinity, 
are two of the main goals of the present analysis.
In the companion paper \cite{LWZ2}, 
we have already made good use of these ideas, applying and further extending them\footnote{
Among other things, analyzing 
nontrivial turning points, finite and infinite,
the case to which we here reduce.}
to obtain the result of high-frequency stability of 
detonation waves in certain media, 
making rigorous the important observations of Erpenbeck \cite{Er} 
made by a combination of formal and rigorous analysis
in the 1960's, but up to now not rigorously verified.
At the same time, 
given the sometimes bewildering array of different techniques that
have been developed for the analysis of this problem, including asympotic ODE and microlocal analysis/WKB expansion, 
and $C^r$ vs. analytic stationary phase, 
we seek
to make clear what can and cannot be accomplished
under various assumptions on \eqref{semi};
that is, to remove the uncertainty whether a stronger result could perhaps be obtained by a different technique.

Here, our main result is to make an explicit connection between existence of block-diagonalizing transformations 
and decay rates for certain oscillatory integrals
\be\label{osc}
\lim_{h\to 0^+} \int_a^b e^{\frac {\phi(y,h)}{h}} a(y,h)dy,
\ee
whereby we are able to resolve a number of such questions by stationary phase computations under appropriate conditions on 
the symbol $a$.
In particular, we show that: (i) block-diagonalization can in general be done only {\it locally},
answering a longstanding question posed by Wasow in his 1985 text \cite{W},
and (ii) in general requires {\it analyticity} and not just $C^r$ or $C^\infty$ of the coefficients of \eqref{semi}.

The former is discussed in the Remark, p. 89 of \cite{W}, comparing analogs of our Theorem \ref{wasow1} to an analog
of our Theorem \ref{t:approx}:
{\it 
``Theorems 6.1-1 and 6.1-2 are strictly local, although the decomposition described in Theorem 12.3-1 of the Appendix is globally valid in all of D.  A global uncoupling of the given differential equation by one and the same transformation with an asymptotic series in powers of epsilon, valid in large regions would be a boon to the theory.  On the other hand it is quite possible that such a theorem does not exist, and then one would like to see counterexamples." 
}

The latter appears to be linked to interesting recently observed phenomena in spectral theory \cite{HS} (almost-sure 
diffusion of spectra under random $C^\infty$ perturbation of an analytic-coefficient operator) and propagation of
singularities \cite{Leb} (diffraction by $C^\infty$ vs. analytic boundary in $\R^3$).
It is obtained via sharp stationary phase estimates \eqref{element}--\eqref{comp} for {\it Gevrey class} symbols $a$,
interpolating between the algebraic van der Korput bounds for $C^r$ symbols and the exponential bounds for analytic $a$;
itself of independent interest, {this estimate too so far as we know is new}.

\medskip

{\it Notation.} 
Symbols $\sim$, $\lesssim$, $\gtrsim$ indicate equality/inequality up to a 
constant factor bounded uniformly with respect to parameters.
$\sigma(M)$ indicates spectrum of a matrix or linear operator $M$.

\subsection{Background/previous results}\label{s:back}
Before stating our main results, we set the stage with a brief further discussion
of some background and motivation for the analysis.

\subsubsection{Block diagonalization and WKB expansion}\label{s:turn}
The classical WKB approach to approximating solutions of 
a system 
\be\label{simple}
hZ'=A(x)Z + hB(x)Z
\ee
falling under the general form \eqref{semi}
is to seek a basis of approximate solutions of form
\be\label{seriesexp}
Z_j(x)= e^{h^{-1} h_j(x) +\sum_{i=0}^j h^i  k^i_j(x)}P^{h,m}(x),
\qquad
P_j^{h,m}(x)= P_{j,0}(x) + h P_{j,1}(x) + \dots + h ^{m}P_{{j,m}}(x),
\ee
$ (h \partial_x - \Phi(x,h))Z_j= O( h^m |Z_j|)$--
equivalently, a parametrix $\Psi^h_m= P^h 
e^{ (h H +K)(y)|_a^x}$,
$H=\diag\{h_j\}$, $K=\diag\{\sum_{i=0}^j h^i k_j^i\}$--
where $a_j=\partial_x h_j$ and $P_{j,0}$ are eigenvalues and eigenvectors of $A(x)$,
and $P^{h,m}:= (P_1^{h,m}, \dots, (P_n^{h,m}) $ denotes the matrix with columns $P_j^{h,m}$.
So long as the eigenvalues $a_j$ remain distinct, so that $P^h$, $(P^h)^{-1}$ may be
taken uniformly bounded, one can convert the formal $h^m$ modeling error
to a rigorous convergence bound by a Lyapunov-Perron type integral equation
mimicking the usual construction of invariant manifolds, in which
$j$th parts of the propagator are integrated against modeling error
along ``progressive contours'' for which $\Re (h_j-h_m)(z)$ 
is nonincreasing for all $m$ \cite{W,O}.
These may be real contours if the eigenvalues
of $A$ maintain a neutral pairwise spectral gap, but in
general are complex, requiring $A$, $B$ analytic.

Difficulties occur at {\it nontrivial turning points}, where eigenvalues of $A$ collide,
{\it and}, for unbounded domains, {\it at $\infty$}, 
where the usual prescription of progressive contours breaks down.
Here, we will seek not to carry out a complete expansion as in \eqref{seriesexp},
but only a block-diagonalization corresponding to invariant subspaces  
with distinct spectra of $A$: essentially a vector version of \eqref{seriesexp}.
This is of course also a preliminary step to full conjugation, decoupling the problem into scalar modes and
irreducible $m\times m$ blocks, $m>1$, containing nontrivial turning points, 
to be analyzed by more special techniques as in, e.g., \cite{W,O}.
Importantly, this includes block-diagonalization at infinity.

\subsubsection{$C^r$ vs $C^\omega$ diagonalization, and diagonalization on unbounded domains}\label{s:challenges}
A new aspect of the 1- and multi-D analyses of detonation stability
in \cite{Z1,LWZ1} was to carry out rigorous WKB-type expansion for \eqref{simple} on unbounded domains. 
This involved (i) effectively resumming the usual series expansion 
to obtain an exact solution at infinity/integrability of modeling error in $x$, and 
(ii) the development in \cite{Z1} of a new
``variable coefficient gap lemma'' for $C^r$-coefficient ODE, generalizing to the semiclassical setting the standard gap lemma of \cite{GZ},
by which one may obtain solutions with desired behavior at infinity so long as the associated eigenvalues of $A(x)$ are (a) semisimple and (b) 
satisfy a neutral numerical range condition roughly corresponding to
a neutral spectral gap (separation of real parts) from other eigenvalues, with $A$ and $B$ converging to their limits at infinity
at $L^1$ rate.
The description is valid globally on the (possibly unbounded) interval where (a)--(b) hold.
However, {both conditions (a)--(b) fail in general for the class of
problems arising in high-frequency stability of detonation waves, as studied, e.g., in \cite{Er,LWZ2}}.

This raises the questions (partially alluded to above) whether: 
(1) the conditions (a)--(b) are indeed necessary for the $C^r$-coefficient problem, or whether there could be provided by different techniques a more general $C^r$-coefficient theorem 
requiring only separation and not spectral gap
of the eigenvalues of $A$, the ``natural'' condition needed for formal WKB expansion, and
(2) the classical local diagonalization results of Wasow \cite{W} (for the analytic-coefficient problem) are sharp, 
or whether they could be extended to a global result valid on the whole interval on which the eigenvalues of $A(x)$ remain
separated.
It is these two questions, and the physical issues originating in detonation that prompt them,
that are the primary practical motivations for our analysis.

\subsection{Main results}\label{s:mainresults} 

\subsubsection{The profile problem, assumptions, and approximate block-diagonalization}\label{s:prof}
Returning to 
the PDE problem 
\eqref{planar}-\eqref{pde}, consider the commonly-occurring case of a front- or pulse-type solution
$
\lim_{z\to \pm \infty}\bar w(z)=w_\pm.
$
Writing the standing-wave ODE as a first-order system
\be\label{twode}
Z'=F(Z), 
\ee
in a phase variable $Z$ consisting of $w$ and appropriate derivatives, we find, so long as $w_\pm$ are nondegenerate 
hyperbolic equilibria, i.e., the Jacobians $dF(w_\pm)$ possess no center subspace, 
that the profile $\bar w$ consists of the projection onto the $w$-coordinate of a profile $\bar Z$, 
$\lim_{z\to \pm\infty} \bar Z(z)=Z_\pm$,
of \eqref{twode}, which in turn corresponds generically to a transversal intersection of the stable manifold at $Z_+$ and the unstable manifold at $Z_-$ of \eqref{twode}.
Assuming that the coefficients of \eqref{pde} are $C^r$, we obtain from this construction a profile $\bar w$ that is $C^{r+1}$ in $x$, and (by standard stable/unstable manifold theorems) converges exponentially in $r$ derivatives to its limits $w_\pm$ as
$x\to \pm \infty$.
Thus, {\it for smooth coefficients $f_\alpha \in C^r$, $C^r$ smoothness and exponential convergence at $\pm\infty$
are natural conditions to impose on $A$ and $B$ in \eqref{semi}.}

For analytic coefficients $f_\alpha\in C^\omega$, we obtain by the same construction a solution $\bar w$ that is analytic for 
all $z\in \R$; however, already at the level of profiles,
the situation is slightly more subtle at $z\to \pm \infty$.
Namely, as we show in Theorem \ref{t:stableman} of Appendix \ref{s:appman}, the stable manifold construction in the
analytic coefficient case yields the much stronger result of existence/analyticity in a {\it wedge}
$$
\Re z\geq 0, \quad  |\Im z|\leq \nu |\Re z|,
$$
with exponential decay $|\bar w(z)|\le C(\tilde \eta)e^{-\tilde \eta|\Re z|}$,
$\nu, \tilde \eta>0$,
and similarly for the unstable manifold at $z\to -\infty$.
{\it Analyticity of $A$ and $B$ on a strip around the real axis and wedges around plus and minus infinity,
with exponential convergence as $\Re z\to \pm \infty$, are thus natural assumptions for \eqref{semi} in the analytic coefficient case.}
This observation so far as we know is new; moreover, 
the strengthened assumptions at $\pm \infty$ turn out to be essential for our treatment of exact block-diagonalization.

Our first main result (following) is that, under the above natural assumptions,
and assuming spectral separation of eigenvalues of the limiting coefficient matrix, there exist {\it global}
approximately block-diagonalizing transformations to all orders, preserving the original assumptions.

\bt [Global approximate diagonalization] \label{t:approx}
Consider a general semiclassical limit problem $hW'=A(x,h)W $, $x\in [a,b]\subset \R$,
$A\in C^r(x,h)$,
such that the eigenvalues of $A(\cdot, 0)$ may be divided into two groups 
separated uniformly in $x,h$ for all $x\in \R$.
Then, for each $1\leq k\leq r$, there exists a uniformly invertible change of coordinates 
$W=\mathcal{T}^k(x,h) W_k$, $\mathcal{T}\in C^r$, yielding the approximately diagonalized system
\ba\label{reducedintro}
hW_k'=\bp A^k_{11}& 0\\0 & A^k_{22}\ep W_k 
+ h^k \bp 0 &  \theta^k_1\\
 \theta^k_2& 0 \ep W,
\ea
$A^k, \theta^k_j \in C^{r-k}$, with $\mathcal{T}^{s+1}= \mathcal{T}^{s} + O(h^{s+1})$.
If $a=-\infty$ or $b=+\infty$ and $A$ is
exponentially converging as $x\to \pm \infty$ in up to $r$ derivatives, uniformly in $x,h$, 
then $A^k$ is exponentially converging and $\theta^k_j$ exponentially
decaying in $r-k$ derivatives as $x\to \pm \infty$, uniformly in $h$.
If, moreover, $A$ is analytic on a strip around the real axis and wedges around plus and minus infinity,
with exponential convergence as $\Re x\to \pm \infty$, then $\mathcal{T}^k$, $ A^k$, $\theta^k_j$ have these properties as well.
\et

\begin{proof}
This is an immediate consequence of Proposition \ref{repeatprop} and Remark \ref{genrmk}, below.
\end{proof}

\br\label{scalarrmk}
For $\sigma(dF(z_\pm))$ real, the proof of Theorem \ref{t:stableman} yields analyticity of profiles on
half-planes $\Re z > M$, $\Re z< -M$, $M>>1$.
This was shown in the scalar case in \cite[Proposition 4.1]{LWZ2} using a direct, Implicit Function Theorem argument.
The example $Z\in \C$, $F(Z)= Z^2-Z$, $\bar Z(x)=-\tanh (x)$ shows that this result is sharp, 
as $\tanh(i\tau)= \tan(\tau)$ has poles at $\tau=\pi/2 + 2\pi j$, $j\in \mathcal{Z}$.
Theorem \ref{t:stableman} extends this to the system case, allowing more general applications;
see for example Remark 2.1 \cite[p. 9]{LWZ2} on treatment of detonations with multi-component reactions.
\er

\subsubsection{Local exact block-diagonalization}\label{s:bd}
Complementing the global approximate diagonalization result of
Proposition \ref{repeatprop}, we have the following {\it local, analytic} exact diagonalization
result,
esentially a
 finite $h$-regularity, finite-accuracy version of Wasow's theorem \cite{W}
in the $h$-analytic case.

\bt[Exact local diagonalization at a finite point]\label{wasow1}
Given an ODE
\ba\label{std}
hW'=\bp A_{11}& 0\\0 & A_{22}\ep W 
+ h^p \Theta W, \qquad p\geq 1,
\ea
$A_{ll}$, $\Theta$ uniformly analytic in $x$ in a complex neighborhood
of $x=x_*$ and continuous in $h$ in a postive real neighborhood of $h=0$,
\emph{with no eigenvalues of $A_{jj}$ in common},
 there exists a coordinate change
$W=T(x,h)Z$, $T=\bp I & h^p \alpha_{12}\\ h^p \alpha_{21}& I\ep$,
such that
$
hZ'=\bp A_{11}+h^{p}\beta_{11}& 0\\0 & A_{22}+h^{p}\beta_{22}\ep Z, 
$
with $\alpha_j$, $\beta_j$
uniformly analytic in $x$ and continuous in $h$ in neighborhoods of $h=0$, $x=x_*$, 
\et

\begin{proof}
See Section \ref{s:exact}.
\end{proof}

An important extension of Theorem \ref{wasow1} for the applications 
we have in mind
is the following result giving
existence of block-diagonalizing conjugators near infinity for a linear system \eqref{std} with
$A_j$, $\Theta$ are analytic in $x$ on the wedge
$W_{M,\beta}: \Re x\geq 0$, $|\Im x|\leq \beta \Re x $,
with 
\be\label{expconv}
|A_j(x,h)-A_j(+\infty,h)|\leq Ce^{-\eta \Re x},
\quad
|\Theta(x,h)|\leq C \; \hbox{\rm on $W_\beta$},
\ee
for some limiting value $A_j(+\infty)$ and constant $C$, uniformy in $h<<1$.

\bt [Exact block diagonalization at infinity] \label{singprop}
For $A_{jj}$, $\Theta$ as in \eqref{expconv},
if limits $A_{jj}(+\infty)$ have no eigenvalues in common, then
there exists a uniformly bounded analytic conjugator $T(\cdot)$
on possibly smaller wedge $W_{M',\beta'}$, for $\beta'$ sufficiently small
and $M'>0$ sufficiently large,
with $|T-\Id|=O(h^p)$.
Moreover, if $\Theta $ is uniformly exponentially convergent
as $Re (x)\to \infty$, then $T$ is as well, and if $\Theta$ is
uniformly exponentially decaying as $\Re (x)\to \infty$, then 
$(T-\Id)$ is also.
\et

\begin{proof} See Section \ref{s:conjinfty}.  \end{proof}

So far as we know, both statement and proof of Theorem \ref{singprop} are new; indeed, as mentioned
earlier, the novel assumption of analyticity on a wedge appears to be essential.
A related problem is existence of block-diagonalizing conjugators
for a {\it singular} ODE system
\ba\label{sing}
zh\frac{dW}{dz}=  \bp A_{11}& 0\\0 & A_{22}\ep(z,h) W 
+ h^p \Theta(z,h) W
\ea
in the vicinity of the singular point $z=0$.
Such problems arise, 
for example, in the treatment of
``hybrid resonance'' or ``X-mode'' heating of fusion plasma \cite{DIW,DIL},
where coincidence of regular-singular and turning points lead to interesting physical phenomena.
A first step in their rigorous analysis
is block reduction to a $2\times 2$ system equivalent to a modified Bessel equation \cite{O}.

\begin{corollary}[Exact block diagonalization at a singular point]\label{singprop2}
	Let $A_{jj}$, $\Theta$ in \eqref{sing} be analytic in $z$ and continuous in $h$ on $B(0,r) \times (0,h_0)\subset \C\
	\times \R^+$ and
$A_{11}(0)$, $A_{22}(0)$ have no eigenvalues in common.
Then, there exists a 
uniformly bounded block-diagonalizing conjugator with uniformly bounded inverse
$T=\Id +O(h^p)$ of \eqref{sing}, analytic in $z$ on a slit ball around $z=0$ with branch cut along
the negative real axis, and continuous at $z=0$.
\end{corollary}

\begin{proof} This follows from Theorem \ref{singprop} via the transformation
$z\to x=-\ln z$; see Section \ref{s:conjsing}.  \end{proof}

\br\label{coordrmk}
Example \ref{sliteg} below shows that 
the conclusions of Corollary \ref{singprop2}
are sharp even for $A_{jj}$, $\Theta$ independent of $h$ and analytic at $z=0$;
specifically, there need not then exist a diagonalizer that is analytic on a neighborhood of $z=0$,
despite a formal power series solution to all orders.
At the same time, the hypotheses of analyticity in $z$ on $B(0,r)$ may be weakened to 
analyticity in a neighborhood of the origin on the Riemann surface with branch cut along the negative real axis.
\er

\subsubsection{Oscillatory integrals and counterexamples}\label{s:counter}
Consider the $2\times 2$ triangular system
\ba\label{ntrisys}
hW'=
\mathcal{A}(x,h)W:=
\bp \lambda_1(x)& h^p \theta(x)\\0 & \lambda_2(x)\ep W , \qquad W\in \C^2,\; p\geq 1,
\ea
$\theta$ uniformly bounded, with globally separated eigenvalues $\lambda_1(x)= x+i$, $\lambda_2=-(x+i)$.


\bl\label{l:oscequiv}
There exists $T(x,h)$ on $[-L,L]\subset \R$, $0\leq h\leq h_0$,
$T,$ $T^{-1}$ uniformly bounded in $C^1$, for which 
$W=TZ$ converts \eqref{ntrisys} to a diagonal system $hZ'=D(x,h) Z$, if and only if
\be\label{e:onlyif}
\hbox{\rm $\int_{-x}^x e^{-y^2/h- 2iy/h} \theta(y) dy \lesssim h e^{-x^2/h}$
for all $|x|\leq L$.  }
\ee
\el

\begin{proof}
See Section \ref{s:tri}.
\end{proof}

Through condition \eqref{e:onlyif}, we obtain the following counterexamples.

\bc[Failure of global conjugators]\label{c:global}
For $\theta\not \equiv 0$ analytic on 
$[-L,L]\times [-i,i]$, \eqref{ntrisys} possesses a 
uniformly bounded $C^1$ conjugator on $[-L,L]$ as $h\to 0^+$ if $L<1$ and only if $L\leq 1$;
for $L=1$, it possesses a uniformly bounded conjugator on $[-L,L]$ if and only if $\theta(-i)=0$.
\ec

\bc[Failure of local conjugators for $C^\infty$ coefficients]\label{c:cinfty}
Let $\theta \in C^\infty$ be given by  $\theta(x)=e^{-x^{-\theta}}$ for $x>0$ and $0$ for $x\leq 0$, $\theta>0$. 
Then, \eqref{ntrisys} possesses no uniformly bounded $C^1$ conjugator on any interval $[-L,L]$, $L>0$. 
\ec

Corollaries \ref{c:global} and \ref{c:cinfty} follow in turn from the following estimates proved in Section \ref{s:someosc}.

\bl\label{quadstat}
For $a\not \equiv 0$ analytic on $[-L,L]\times [-i,i]$, and $h\to 0^+$, 
\be\label{e:quadtrans}
\int_{-x}^x e^{-y^2/h- 2iy/h} a(y) dy 
\begin{cases}
	\lesssim h e^{-\frac{x^2}{h}}, & 0<x\leq L < 1,\\
	\sim h^{1/2}a(-i) e^{-\frac{1}{h}} + O(h e^{-x^2/h}) , &  |x|\leq L=1,\\
	\sim h^{(j+1)/2} e^{-\frac{1}{h}}, & 1< c_0\leq x\leq L,\\
\end{cases}
\ee
where $j$ is the order of the first nonvanishing derivative of $a$ at $z=i$.
\el

\bl\label{halfprop}
For $0<c_0\leq x\leq \infty$, 
$\theta= \frac{1}{s-1}\in (0,+\infty)$,
$a(y):= e^{-y^{-\theta}}$ for $y>0$ and $0$ for $y\leq 0$,
\be\label{element}
\int_{-x}^x e^{-y^2/h-2iy/h} a(y) dy \sim
h^{1-1/2s} e^{\frac{-c(s)+ d(s) h^{1-1/s} + O(h^{2(1-1/s)}) }  {h^{1/s}} } , \quad 1<s<\infty,
\ee
as $h\to 0^+$, where $c(s)>0$ and $\Re d(s)=-\cos(\pi(1-1/s))$ is $<0$ for $s<2$.
\el

The symbol $a(y):= e^{-y^{-\theta}}$ for $y>0$, $a(y)=0$ for $y\leq 0$ 
is of Gevrey class $\mathcal{G}^{s,T}$ \cite{Le}, \cite[Rmk. 1.3, p. 3]{KV},
defined by boundedness of the Gevrey norm
$
 \| a\|_{s,T}:= \sup_{j}|\partial_x^j a| (j!)^s/T^{j}
$
for some $T$, with $s=1$ corresponding to analyticity on a strip of width $T$ about the real axis $\R$.
The contrast between \eqref{quadstat} and \eqref{element}
reflects a difference in stationary phase-type estimates for analytic vs. $C^\infty$ symbols $a$,
as quantified in the following more general observation, of interest in its own right.

\bpr\label{p:gevrey}
For $a\in \mathcal{G}^{s,T_0}$ on $[-L,L]$, $T_0, T>1$, 
$|x|\leq L$, and some $c=c(T_1,T,s)>0$,
\be\label{comp}
\int_{-x}^x e^{-y^2/h- 2iy/h} a(y) dy \lesssim h^{1/2} \|a\|_{T,s}  e^{-c/h^{1/s}}.
\ee
\epr

Proposition \ref{p:gevrey} interpolates between the algebraic $O(h^{r})$ van der Korput bounds for $C^r$ symbols 
(roughly, $s=\infty$) and the exponential $O(h^{1/2}e^{-1/h})$ bounds for analytic symbols $a$ obtained by the saddlepoint method/analytic stationary phase, as described in Appendix \ref{s:stat}; so far as we know, this observation also is new.
The lower bounds of Lemma \eqref{element} show that \eqref{comp} is sharp.
Specifically, for $s\geq 2$, $\frac{h^\alpha}{h^{1-\alpha}}= h^{1-2/s}\leq 1$, yielding
$ \int_{-x}^x e^{-y^2/h-2iy/h} a(y) dy \sim h^{1-1/2s} e^{-c(s) /  h^{1/s} }; $
for $1<s<2$, the bound is sharp up to the slower-decaying exponential factor
$e^{( d(s) h^{1-1/s} + O(h^{2(1-1/s)}) )   /h^{1/s} }$.

\br[Failure for analytic projectors]\label{projrmk}
Replacing the diagonal entries of \eqref{ntrisys} by $\lambda_1=x$, $\lambda_2=-x$,
we find by the same argument as for Lemma \ref{l:oscequiv} that diagonalization is possible on $[-L,L]$ if and only if
$\int_{-x}^x e^{-y^2/h} \theta(y) dy \lesssim h e^{-x^2/h}$
for all $|x|\leq L$, which clearly fails.  
Thus, \emph{even in the case that analytic projectors persist,} failure of spectral separation can lead to nonexistence
of an exact diagonalizing transformation.
Note that the radius $|x|=1$ of existence of diagonalizing transformations in Corollary \ref{c:global} corresponds in the complex plane to the radius at which there appears a point $z=-i$ at which $\lambda_1=\lambda_2$ and separation fails,
which is simultaneously a stationary point for the phase $\phi=\int(\lambda_1-\lambda_2)$ appearing in \eqref{e:onlyif}.
Interestingly, the condition $\theta(-i)=0$ determining extensibility up to radius $1$ is the condition that
$\mathcal{A}(-i,h)$ in \eqref{ntrisys} be diagonalizable.
\er

\br\label{eqrmk}
It is an interesting question whether there holds a general lower bound 
in \eqref{comp}, in which case \eqref{comp} would represent an alternative characterization of 
Gevrey class in terms of the F.B.I. transform, analogous to characterizations in terms of the Fourier transform
as, e.g., in \cite{FT}.
\er


\subsection{Conclusions}\label{s:conclusion}
The above results determine what is 
possible in various settings in terms of approximate or exact block-diagonalization, 
in particular settling the two open problems posed by Wasow in his 1965 text \cite{W}
on asymptotic ODE whether global exact diagonalization is possible and whether analyticity of coefficients is necessary.
Namely, we see that, in general, even in the simplest case of a system $hZ'=A(x)Z + hB(x)Z$ 
for which $A$ possesses globally separated eigenvalues,
there exists (only) a finite collection of exact {\it locally diagonalizing } transformations
on neighborhoods covering the domain $[a,b]$ of definition of $A$, $B$,
and this holds in general {\it only for $A$ and $B$ analytic}.

\medskip
{\bf Acknowledgment.}
Thanks 
to Gilles Lebeau and Jean-Marc Delort for stimulating conversations.  
Thanks to
University of Indiana, Bloomington, University of North Carolina Chapel Hill, Universities of Paris 7 and 13, ENS Ulm,
and the Fondation Sciences Math\'ematiques de Paris for their hospitality during visits in which this research was partially carried out.

\section{Repeated diagonalization and exact local block-diagonalization}

In this section, we compare two methods for obtaining a block-diagonal system from a given semiclassical system of ODE,
the first approximate, and the second exact.
The first can be used to compute the second to arbitrary order, while the second gives rigorous validation to the first.
The approximate block-diagonaization is done globally; the exact block-diagonalization is local.
Novel aspects of our analysis are the treatment of the point at infinity and of regular-singular points.

\subsection{The method of repeated diagonalization}\label{s:repeat}
We start by recalling the method of repeated
diagonalization as implemented in \cite{MaZ},
by which one may obtain from an approximately block-diagonal system
with \emph{spectrally separated blocks}
a series of approximately block-diagonal systems
of successively higher accuracy.
For related methods, see, for example, \cite{L,F,W,E,BEEK}.
Consider an approximately block-diagonal ODE 
in the semiclassical limit $h\to 0^+$:
\ba\label{stdr}
hW_j'=\bp A^j_{11}& 0\\0 & A^j_{22}\ep W_j 
+ h^j \bp 0 &  \theta^j_1\\
 \theta^j_2& 0 \ep W_j,
\qquad x\in [a,b]\subset \R.
\ea

\bpr[Adapted from \cite{MaZ}]\label{repeatprop}
Let $A^j_{ll}(x,h), \theta^j_l (x,h)\in C^r(x)$ and $C^0(h)$ on $[a,b]\times [0,h_0]$, uniformly
in both coordinates, with $r>j\geq 1$, $l=1,2$.
Suppose, moreover, that $A^j_{ll}$ have no eigenvalues in common.  Then, 
for $j\leq k\leq r+j$, $l=1, 2$, and $h>0$ sufficiently small, 
there exists series of transformations $W_{k}=T_k W_{k+1}$,
$
T_k:=\bp \Id & h^kc^k_1\\
 h^kc^k_2 & \Id \ep
\in C^{r+j-k}$,
converting \eqref{stdr} to
\ba\label{reduced}
hW_k'=\bp A^k_{11}& 0\\0 & A^k_{22}\ep W_k 
+ h^k \bp 0 & h^k \theta^k_1\\
h^k \theta^k_2& 0 \ep W.
\ea
If $a=-\infty$ or $b=+\infty$ and $\theta^j_l$ are exponentially decaying in up to $r$
derivatives as $x\to \pm \infty$, uniformly in $h$, then $\theta^k_l$ are 
exponentially decaying in up to $r+j-k$ derivatives as $x\to \pm \infty$, uniformly in $h$.
If $A$ is analytic on a strip around the real axis and wedges around plus and minus infinity,
with exponential convergence as $\Re x\to \pm \infty$, then $T_k$, $ A_k$, $\theta^k_j$ have these properties as well.
\epr

\begin{proof} We proceed by induction from $k=r+1$ up to $K$, at each step defining
$T_k$ such that 
\be\label{blockdiag}
\hbox{\rm $D_k:=T_k^{-1}A_{k-1}T_k$ is block-diagonal,}
\ee
and setting 
\ba\label{approxdefs}
A_k&= T_k^{-1}A_{k-1}T_k - h\, \blockdiag \{ T_k^{-1}\partial_x T_k \},\\
\theta_k&= h^{1-k} \Big( -T_k^{-1}\partial_x T_k + \blockdiag \{ T_k^{-1}\partial_x T_k \} ,
 \ea
 from which we obtain, evidently, \eqref{reduced}.

 To complete the proof, it remains to show that, for $\theta$ bounded and $h>0$ sufficiently small,
 \eqref{blockdiag} has a unique solution of form
 $T_k=\bp \Id & \theta_1^k\\ \theta_2^k& \Id \ep$, depending in $C^1$ fashion on $\theta$.
	A straighforward calculation equating first diagonal, then off-diagonal blocks
	in the equation $T_kD_{k-1}= A_{k-1} T_k$, yields the equivalent system
		\ba\label{nricatti}
		A^{k-1}_{11}\alpha^k_{12}- \alpha^k_{12}A^{k-1}_{22} +  
		\Theta^{k-1}_{12}- h^{2k}\alpha^k_{12} \Theta^{k-1}_{21}\alpha^k_{21}- h^k \Theta^{k-1}_{11} \alpha^k_{21}&=0 ,\\
		A^{k-1}_{22}\alpha^k_{21}- \alpha^k_{21}A^{k-1}_{11} +  
		\Theta^{k-1}_{21}- h^{2k}\alpha^k_{21} \Theta^{k-1}_{12}\alpha^k_{12}- h^k \Theta^{k-1}_{22} \alpha^k_{21}&=0 ,\\
\ea
or, written in block vector form in terms of $\alpha=(\alpha_{12},\alpha_{21})$,
\be\label{nvecversion}
\mathcal{F}(\alpha_k, \Theta_{k-1}, h):= \cA \alpha_k +  Q(\alpha_k,\Theta_{k-1},h)=0,
\ee
where 
\be\label{Adef}
\cA \bp \alpha_{12}\\ \alpha_{21} \ep:= \bp A^{k-1}_{11} \alpha_{12}- \alpha_{12} A^{k-1}_{22}\\ 
A^{k-1}_{22} \alpha_{21}- \alpha_{21} A^{k-1}_{11}\ep ,
\ee
\be \label{nQbds}
Q(\alpha,\Theta,h)=O(|\Theta|)(1+h^k|\alpha|+h^{2k}|\alpha|^2),
\qquad
Q_\alpha(\alpha,\Theta,h)=O(|\Theta|)(h^k +h^{2k}|\alpha|) .
\ee

From \eqref{Adef}--\eqref{nQbds}, we have $\mathcal{F}(0,0,h)\equiv 0$,
while, assuming uniform boundedness and uniform separation of the spectra of 
$A^{k-1}_{11} $ and $A^{k-1}_{22} $, the decoupled linear operator
$\partial_\alpha\mathcal{F}(0,0,h)=\mathcal{A}$ is uniformly invertible, whence we obtain by the Implicit
Function Theorem existence of a unique small solution $\alpha_k=\mathcal{G}(\Theta_{k-1},h)$, smooth in both variables. 
Noting that the properties of uniform boundedness and uniform separation of the spectra of 
$A^{k}_{11}=A^{k-1}_{11}+O(h^k) $ and $A^{k}_{22}= A^{k-1}_{22}+ O(h^k) $ persist (by smallness of $h$) throughout
the iteration, we are done.
\end{proof}

\br\label{concatrmk}
The repeated diagonalization expansion may be seen to be a block-diagonal
version of the classical WKB expansion \eqref{seriesexp},
with $P^{h,m}$ analogous to the concatenation 
\be\label{Tcon}
\mathcal{T}_{m,h}:=
T_m\cdot T_{m-1}\cdot \dots T_{j+1}= \mathcal{T}_{m-1,h}+O(h^m).
\ee
By Remark \ref{genrmk}, the process can be repeated until the system 
approximately decouples into distinct blocks whose eigenvalues all collide: 
that is, which are either scalar or else possess nontrivial turning points.
In the case that the initial coefficient matrix has everywhere 
distinct eigenvalues, the result is a complete decomposition into scalar modes analogous to the WKB expansion \eqref{seriesexp}.
\er

\br\label{genrmk}
A general semiclassical limit problem $ hW'=A(x,h)W $
may, by standard spectral perturbation theory \cite{K},
be converted to form \eqref{stdr}, $j=1$, by an initial block-diagonalizing
transformation, $W =\hat TW_1$ such that $\hat T^{-1}A\hat T=\bp A_{11} & 0\\ 0 & A_{22}\ep$,
yielding \eqref{stdr} with $A^1_{jj}=A_{jj}-h(\hat T^{-1}\hat T_x)_{jj}$,
$\theta_1^1= -(\hat T^{-1}\hat T_x)_{12}$, $\theta_1^2= -(\hat T^{-1}\hat T_x)_{21}$,
so long as there exist two groups of eigenvalues of 
$A$ that remain separated uniformly in $x,h$.
Specifically, we may take $\hat T=(\hat T_1, \hat T_2)$, where the columns of $\hat T_j$ are bases of the
associated total eigenspaces of these two groups, defined by Kato's ODE 
\be\label{KODE}
(d/dx)\hat T_j= [\Pi_j,(d/dx)\Pi_j]\hat T_j,
\ee
where $\Pi_j$ denote the associated eigenprojections and $[M,N]:=MN-NM$ the usual matrix commutator \cite{K}.
Thus, there is no loss of generality in starting with the form \eqref{stdr}.
The projectors $\Pi_j$, hence the solutions $\hat T_j$ of the linear ODE \eqref{KODE}, inherit the same regularity in
$(x,h)$ possessed by the original coefficient $A(x,h)$.
Moreover, {\it if $A(\cdot, h)$ exponentially approaches limits as $x\to \pm \infty$, then $\Pi_j$ and thus $\hat T$,
converge at the same rate, and $\theta_j\lesssim (d/dx) \hat T_j$ decay exponentially.}
The transformation $\hat T$ is typically not explicitly 
computable, but may (similarly
as in the WKB expansion of Section \ref{s:turn}) be expressed as a matrix perturbation series in $h$.
\er

\subsection{Exact analytic local block-diagonalization}\label{s:exact}
Exact diagonalization at a finite point follows similarly as in \cite{W}, as we now describe.

\begin{proof}[Proof of Theorem \ref{wasow1}]
	Take without loss of generality $x_*=0$,
	so that we seek a block-diagonalization near $x=0$.
A straighforward calculation equating first diagonal, then off-diagonal blocks
in the equation $(hT'+TD)Z= ATZ$, 
$D=\bp A_{11}+h^p\beta_{11} & 0\\0 &  A_{22}+h^p\beta_{22} \ep$,
		yields Ricatti equations
		\ba\label{ricatti}
h \alpha_{12}'&= A_{11}\alpha_{12}- \alpha_{12}A_{22} +  
		\Theta_{12}- h^{2p}\alpha_{12} \Theta_{21}\alpha_{21}- h^p \Theta_{11} \alpha_{21} ,\\
h \alpha_{21}'&= A_{22}\alpha_{21}- \alpha_{21}A_{11} +  
		\Theta_{21}- h^{2p}\alpha_{21} \Theta_{12}\alpha_{12}- h^p \Theta_{22} \alpha_{21} ,\\
\ea
with
\ba\label{betas}
\beta_{11} &= \Theta_{11} + h^p \Theta_{12}\alpha_{21} ,\\
\beta_{22} &= \Theta_{22} + h^p \Theta_{21}\alpha_{12} .\\
\ea
Viewed as a block vector equation in $\alpha=(\alpha_{12},\alpha_{21})$, \eqref{ricatti} has form
$ h\alpha'= \cA \alpha + Q(\alpha,\Theta,h)$, or
\be\label{vecversion}
 h\alpha'= \cA(0) \alpha +
 (\mathcal{A}(z)-\mathcal{A}(0))\alpha +  Q(\alpha,\Theta,h),
\ee
where
\be\label{calA}
\cA \bp \alpha_{12}\\ \alpha_{21} \ep:= \bp A_{11} \alpha_{12}- \alpha_{12} A_{22}\\ 
A_{22} \alpha_{21}- \alpha_{21} A_{11}\ep ,
\ee
\be \label{Qbds}
Q(\alpha,\Theta,h)=O(|\Theta|)(1+h^p|\alpha|+h^{2p}|\alpha|^2),
\qquad
Q_\alpha(\alpha,\Theta,h)=O(|\Theta|)(h^p +h^{2p}|\alpha|) .
\ee

The eigenvectors of $ \cA(0) $ may be expressed as tensor products 
$\alpha= \bp \phi_1 \tilde \phi_2^*\\ 0\ep$ and $\bp 0\\ \tilde \phi_3  \phi^*_4\ep $
of eigenvectors of $A_{11} $ and
$A_{22}^* $, and the corresponding eigenvalues as the differences in the eigenvalues associated with $\phi_1$, $\phi_2$
and $\phi_3$, $\phi_4$.
Thus, by separation of eigenvalues of $A_{11}$ and $A_{22}$, we find that that 
\emph{$\mathcal{A}(0)$ has no zero eigenvalues.}
It follows that there is $\gamma\in \C$, $|\gamma|=1$, with argument arbitrarily close to zero,
for which \emph{$\gamma \mathcal{A}(0)$ has no center subspace.}
Denoting by $\Pi_U$/$\Pi_S$ the unstable/stable projectors of $\gamma \mathcal{A}(0)$, we thus have, for
$z$ with direction sufficiently close to that of $\pm \gamma$ and $\eta, C>0$: 
\ba \label{expbd}
|e^{h^{-1}\cA(0) z}\Pi_S| &\le Ce^{-h^{-1}\eta \Re z } \quad \hbox{\rm for $\Re z \geq 0$,}\\
|e^{h^{-1}\cA(0) z}\Pi_U| &\le Ce^{h^{-1}\eta \Re z }, \quad \hbox{\rm for $\Re z \leq 0$.}
\ea

Defining now $\Pi_U\alpha=0$ at $z_*:= -M \gamma$ and  $\Pi_S\alpha=0$ at $z^*:= M \gamma$ for 
$M>0$ real,
we obtain by Duhamel's principle the integral fixed-point equation (suppressing dependence of $\alpha $ on $h$):
\ba \label{inteqn}
\alpha(x)=  \CalT \alpha(x)&:=
  h^{-1}\int_{z_*}^x e^{h^{-1}\cA(0) (x-y)}
  \Pi_U \big( (\mathcal{A}(y)- \mathcal{A}(0))\alpha(y) + Q(\alpha,\Theta,h)(y) \big) \,dy \\
 &\quad +  h^{-1}\int_{z^*}^x e^{h^{-1}\cA(0) (x-y)}
  \Pi_S \big( (\mathcal{A}(y)- \mathcal{A}(0))\alpha(y) + Q(\alpha,\Theta,h)(y) \big) \,dy 
,
\ea
defined on the diamond 
\be\label{diamond}
\mathcal{D}:=\{ x:\, |\arg \big((x-z_*)/\gamma\big)|, \; |\arg \big( (z^*-x)/\gamma \big)|\leq \eps \}
\ee
containing a neighborhood of the origin, with $0< \eps, M \ll 1$.
Noting that $(x-y)$ has angle arbitrarily close to that of $\pm \gamma$ for $x\in \mathcal{D}$ and $y\in [z_*,x],
[x,z^*]$ for $\eps >0$ sufficiently small, 
we have by \eqref{expbd} the bounds
	$|e^{h^{-1}\cA(0) (x-y)}\Pi_S| \le Ce^{-h^{-1}\eta \Re (x-y) }$ and
	$|e^{h^{-1}\cA(0) (x-y)}\Pi_U| \le Ce^{h^{-1}\eta \Re (x-y) }$ 
in \eqref{inteqn},
whence we obtain by integrability of $h^{-1}e^{-h^{-1} \eta t }$ over $t\in \R^+$ the estimates
\ba\label{contest}
\|\CalT(\alpha)\|_{L^\infty(\mathcal{D})}& \leq
C \|\mathcal{A}-\mathcal{A}(0)\|_{L^\infty(\mathcal{D})} \|\alpha\|_{L^\infty(\mathcal{D})} \\
&\quad +
C\|\Theta\|_{L^\infty(\mathcal{D})}
\Big( 1 +  h^p \|\alpha\|_{L^\infty(\mathcal{D})} 
+h^{2p} \|\alpha\|_{L^\infty(\mathcal{D})}^2\Big),\\ 
\|\CalT(\alpha_1-\alpha_2)\|_{L^\infty(\mathcal{D})}& \leq
C\Big( \|\mathcal{A}-\mathcal{A}(0)\|_{L^\infty(\mathcal{D})} 
+ \|\Theta\|_{L^\infty(\mathcal{D})} \big( h^p +h^{2p} \sup_j \|\alpha_j\|_{L^\infty(\mathcal{D})}\big)\Big)\\
&\quad \times
\|\alpha_1-\alpha_2\|_{L^\infty(\mathcal{D})}.
\ea
Taking $M$, hence $|\mathcal{A}-\mathcal{A}(0)|$, and $h$ sufficiently small, we find from \eqref{contest}(i)
that $\mathcal{T}$ takes the ball $ B(0, 2C\|\Theta\|_{L^\infty(\mathcal{D})}) $ to itself, and from 
\eqref{contest}(ii) that 
$\mathcal{T}$ is contractive on $ B(0, 2C\|\Theta\|_{L^\infty(\mathcal{D})}) $ with contraction constant $<1/2$. 
It follows that \eqref{inteqn} determines
a unique solution for $M, \eps \ll 1$, which,
moreover, is bounded as claimed.
Regularity with respect to parameters is inherited as usual
through the fixed-point construction via the Implicit Function Theorem.
\end{proof}

\subsection{Exact block-diagonalization at infinity}\label{s:conjinfty}
Exact diagonalization at infinity follows similarly, but with an important modification having to do with restriction to a wedge.
The argument fails for functions merely analytic on a strip; analyticity on a full wedge is needed
in the proof.

\begin{proof}[Proof of Theorem \ref{singprop}]
Under assumptions \eqref{expconv}, consider the equation
\ba \label{singinf}
h\frac{dW}{dx}=  \bp A_{11}& 0\\0 & A_{22}\ep(x,h) W 
+ h^p \Theta( x,h) W.
\ea
Following the same steps as in the proof of Theorem \ref{wasow1},
we obtain as in \eqref{vecversion} equation
\be\label{vecinfty}
h \alpha'= \cA(\infty) \alpha +O(|\cA(z)-\cA(\infty|)\alpha +  Q(\alpha,\Theta) 
\ee
for entries $\alpha=(\alpha_{12},\alpha_{21})$
of a diagonalizing transformation
$
T=\bp I & h^p\alpha_{12}\\ h^p \alpha_{21}& I\ep,
$
$\mathcal{A}$ as in \eqref{calA}.

Again, we note that the eigenvalues of $\mathcal{A}(\infty)$, formed 
by differences in eigenvalues of $A_{11}(\infty)$ and $A_{22}(\infty)$, do
not include the value zero.
Choosing $\eps>0$ so that there are no eigenvalues of $\mathcal{A}(\infty)$ on the rays with angle $\pm (\pi/2+\eps)$,
we can thus divide the eigenvalues of $\cA(\infty)$ among three subsets 
in a way that that persists under small variations in $h$ (or the suppressed
parameter $q$ of \eqref{semi}), namely:
I. eigenvalues with argument lying strictly between 
$\pi/2+ \eps$ and $3\pi/2 -\eps$;
II. eigenvalues with argument strictly between $- \eps$ and  $\pi/2+\eps$;
and 
III.  eigenvalues with argument strictly between $-\pi/2-\eps$ and  $ \eps$,
with associated projectors $\Pi_I$, $\Pi_{II}$, and $\Pi_{III}$.

Taking now $\eps>0$ sufficiently small,
and defining the wedge
\be\label{wedge}
\mathcal{W}_{M',\eps}: \{x: -\eps \leq \arg (x-M')\leq \eps\},
\ee
and direction vectors $\gamma_\pm := e^{\pm i\eps}$,
and taking $M'>0$ sufficiently large, we can, similarly as in the finite case, 
obtain $\alpha$ as the unique solution in $L^\infty(\mathcal{W}_{M',\eps})$, of the fixed-point equation
\ba \label{fixedinfty}
\alpha(x)=  \CalT \alpha(x)&:=
  h^{-1}\int_{M'}^x e^{h^{-1}\cA(\infty) (x-y)}
\Pi_I \big( (\mathcal{A}(y)-\mathcal{A}(\infty)\alpha))+  Q(\alpha,\Theta,h)(y) \big) \,dy \\
&\quad +  h^{-1}\int_{x+\gamma_-( +\infty) }^x e^{h^{-1}\cA(\infty) (x-y)}
\Pi_{II} \big( 
(\mathcal{A}(y)-\mathcal{A}(\infty)\alpha)
+  Q(\alpha,\Theta,h)(y) \big) \,dy \\
&\quad +  h^{-1}\int_{x+\gamma_+(+\infty)}^x e^{h^{-1}\cA(\infty) (x-y)}
\Pi_{III} \big( (\mathcal{A}(y)-\mathcal{A}(\infty)\alpha) +  Q(\alpha,\Theta,h)(y) \big) \,dy \, , \\
\ea
where the contour integrals are understood to be along straight lines.

For, noting that multiplication of $\mathcal{A}(\infty)$ by $\gamma_\pm$ rotates its spectrum by angle $\pm \eps$, 
and recalling that eigenvalue in region II have angle strictly between $-\eps $ and $ \pi/2 + \eps$, 
we have that $\gamma_-\mathcal{A}(\infty) \Pi_{II}$ restricted to range of $\Pi_{II}$ has eigenvalues of angle
strictly between $-2 \eps$ and $\pi/2$ and, similarly,
$\gamma_+\mathcal{A}(\infty)\Pi_{III}$ restricted to range of $\Pi_{III}$
has eigenvalues of angle strictly between $-\pi/2 $ and $2\eps$, whence for $\eps$
less than $\pi/2$ both have eigenvalues of strictly postive real part. 
Likewise, for $-\eps\leq \theta\leq \eps $, 
$e^{i\theta} \mathcal{A}(\infty)\Pi_{I}$ restricted to range of $\Pi_{I}$ has eigenvalues of strictly negative real part.
It follows that, analogously to \eqref{expbd},
\ba \label{2expbd}
|e^{h^{-1}e^{i\theta} t \cA(0) }\Pi_I| &\le Ce^{-h^{-1}\eta t } \quad \hbox{\rm for $|\theta|\leq \eps$, $t \geq 0$ real,}\\
|e^{h^{-1}\gamma_+ t \cA(0) }\Pi_{II}| &\le Ce^{h^{-1}\eta t }, \quad \hbox{\rm for $t\leq 0$ real,}\\
|e^{h^{-1}\gamma_- t \cA(0) }\Pi_{III}| &\le Ce^{h^{-1}\eta t }, \quad \hbox{\rm for $t\leq 0$ real},\\
\ea
for some $C, \eta>0$, yielding analogously to \eqref{contest}
\ba\label{ncontest}
\|\CalT(\alpha)\|_{L^\infty(\mathcal{W}_{M',\eps})}& \leq
C \|\mathcal{A}-\mathcal{A}(\infty)\|_{L^\infty(\mathcal{W}_{M',\eps})} \|\alpha\|_{L^\infty(\mathcal{W}_{M',\eps})} \\
&\quad +
C\|\Theta\|_{L^\infty(\mathcal{W}_{M',\eps})}
\Big( 1 +  h^p \|\alpha\|_{L^\infty(\mathcal{W}_{M',\eps})} 
+h^{2p} \|\alpha\|_{L^\infty(\mathcal{W}_{M',\eps})}^2\Big),\\ 
\|\CalT(\alpha_1-\alpha_2)\|_{L^\infty(\mathcal{W}_{M',\eps})}& \leq
C\Big( \|\mathcal{A}-\mathcal{A}(\infty)\|_{L^\infty(\mathcal{W}_{M',\eps})} 
+ \|\Theta\|_{L^\infty(\mathcal{W}_{M',\eps})} \big( h^p +h^{2p} \sup_j \|\alpha_j\|_{L^\infty(\mathcal{W}_{M',\eps})}\big)\Big)\\
&\quad \times
\|\alpha_1-\alpha_2\|_{L^\infty(\mathcal{W}_{M',\eps})}.
\ea
Taking $h$ sufficiently small and $M'$ sufficiently large, hence $|\mathcal{A}-\mathcal{A}(\infty)|$ sufficiently small, 
we find from \eqref{ncontest}(i)
that $\mathcal{T}$ takes the ball $ B(0, 2C\|\Theta\|_{L^\infty(\mathcal{W}_{M',\eps})}) $ to itself, and from 
\eqref{ncontest}(ii) that 
$\mathcal{T}$ is contractive on $ B(0, 2C\|\Theta\|_{L^\infty(\mathcal{W}_{M',\eps})}) $ with contraction constant $<1/2$. 
It follows that \eqref{inteqn} determines
a unique solution for $M, \eps \ll 1$, which,
moreover, is bounded as claimed.
Regularity with respect to parameters is inherited through the fixed-point construction.
Finally, straightforward estimates using exponential convergence (decay) 
of $\Theta$  in $y$, and exponential decay of $|\mathcal{A}(y)-\mathcal{A}(\infty)|$
and propagators $e^{h^{-1} \cA(\infty) (x-y)}$ in $|\Re(x-y)|$
yield exponential convergence (decay) of $\alpha$.
\end{proof}

\br\label{wedgermk} The introduction of oblique contours $[x, x+\gamma_\pm (+\infty)]$, made possible by the assumption
of analyticity on a wedge,
is what makes possible the subdivision of eigenvalues of $\mathcal{A}(\infty)$ yielding strict exponential decay.
With analyticity only on a strip, we would have a problem with algebraic growth in case that $\mathcal{A}(\infty)$
had a pure imaginary eigenvalue with nontrivial Jordan block.
Less critically, we would have at best neutral decay of propagators, hence would lose a power of $h$ in the exact conjugation,
since $h^{-1}e^{-\eta t/h}$ is integrable for $t\in \R^+$ only for $\Re \eta>0 $ (see \cite{Z1, LWZ1}).
\er

\br\label{nonexprmk}
By $O(e^{-\eta |\Re (x-y)|})$ decay of propagators, the contribution to \eqref{fixedinfty}
from $|x-y|\geq |x|/C$ is negligible for any $C>0$, whence, taking $C\gg 1$, we recover 
from decay (convergence) of $\theta$ decay (convergence) of $\alpha$, whether or not the rate of decay 
(convergence) is exponential.  In particular, for algebraic decay (convergence) of $\theta$, 
we obtain decay (convergence) of $\alpha$ at the same rate.
\er

\subsection{$C^r$ version in the case of a spectral gap}\label{s:gap}
By a modification of the ``variable-coefficient
conjugation lemma,'' Lemma 4.3 of \cite{Z1} (see also Remark 7.2 of the same reference),
we may alternatively obtain existence of a block-diagonalizing 
conjugator 
for $A_{jj}$, $\Theta$
real-valued and merely $C^r$, $r>1$, again on $[M,\infty)$,
$M>>1$,  under the additional assumptions that (i)
$\Re A_{11}\geq \Re A_{22}$ or vice versa, where $\Re N:=(1/2)(N+N^*)$
refers to the symmetric part of a matrix $N$,
and (ii) $\Theta$ is exponentially decaying as $x\to +\infty$, uniformly in $h$.
If $p\geq 2$, we may, further, take $M$ with any value (even 
egative),
for $h>0$ sufficiently small,
where, $A_{jj}$ and $p$ are as in \eqref{std}.
(Recall, Remark \ref{genrmk}, that exponential decay condition
(ii) may be arranged so long as there
is spectral separation (as opposed to gap) between blocks.)
The proof amounts to the observation that we may in this case use the 
fixed-point construction \eqref{inteqn} restricted to $x\in \R$
and $z_*:=-M$, $z^*=M$,
noting that \eqref{expbd} hold with
$\gamma=1$ and $\eta=0$. 
This clarifies perhaps 
what can already be done with $C^r$ coefficients
and what is gained from analyticity.
We show in Section \ref{s:cegs} that these distinctions
are actual and not only apparent; that is, the hypotheses
for analytic vs. $C^r$ versions are sharp.


\br
If an analytic function has an accumulation point of purely 
imaginary values on the real line then it is imaginary-valued on the real 
line (inspection of coefficients at the accumulation point).
Thus,  crossing of real parts are isolated, and one can use the 
real-valued
variable coefficient conjugation lemma to treat intervals between,
so long as the neutral numerical range condition
$\Re A_{11}\geq \Re A_{22}$ is maintained-
roughly, so long as neutral eigenvalues stay semisimple. 
An advantage of this approach is to know the size of intervals on which conjugators
exist.
On the other hand, the $C^r$ conjugators so constructed
are not in general analytic, even if coefficients $A_{jj}$, $\Theta$ 
are analytic, unless
the neutral numerical range condition extends to a complex strip.
\er

\subsection{Diagonalization at a (finite) singular point}\label{s:conjsing}
{ }

\begin{proof}[Proof of Corollary \ref{singprop2}]
	This case may be converted to the problem treated previously in Theorem \ref{singprop}.
In coordinates $x:=-\ln z$, we obtain
\ba\label{singinfeg}
h\frac{dW}{dx}=  \bp A_{11}& 0\\0 & A_{22}\ep(e^{-x},h) W 
+ h^p \Theta(e^{-x},h) W,
\ea
$A_j$, $\Theta$ analytic in $x$ on a half-plane $\Re x\geq M$,
and convergent as $\Re x\to +\infty$ ($ z\to 0$).
Applying Theorem \ref{singprop}, we obtain a conjugator that is
analytic in $x=-\ln z$ on a wedge $x \in \mathcal{W_{M',\eps}}$ as in \eqref{wedge} for $M'>0$
sufficiently large and $\eps>0$ sufficiently small and convergent
as $\Re x\to \infty$,
hence analytic in $z$ on a slit disk about $z=0$ with cut along the negative
real axis and continuous at $z=0$.
\end{proof}



\section{Oscillatory integral estimates}\label{s:someosc}
As explored further in Section \ref{s:tri}, 
existence of conjugators is related to decay rates for certain oscillatory integrals.
For later use, we carry out here some estimates needed for our analysis.

\subsection{Variation on a classical sum}\label{s:classic}
We first generalize the estimate
$ \int_{-\infty}^{+\infty} e^{-\frac{y^2}{h}} e^{\frac{-2iy}{h}}dy \sim h^{1/2}e^{-1/h}, $
obtained by direct Fourier transform to general analytic symbol $a(y)$ and bounded domains.

\begin{proof}[Proof of Lemma \ref{quadstat}]
	Consider $ \int_{-x}^{x} e^{-\frac{y^2}{h}} e^{\frac{-2iy}{h}}a(y)dy $, where $a$ is
	holomorphic on $[-L,L]\times [-i,i]$.
	For $x\geq c_0 > 1$, we obtain the estimate by the Analytic Stationary Phase Lemma \ref{l:stat}
	together with Remark \ref{r:stat}(i),
	defining the contour $\mathcal{C}= [x,-i+\eps]\cup [-i+\eps, -i-\eps]\cup [-i-\eps,-x]$, 
	$0<\eps\ll 1$ real, and noting that $[-x,x]\cup \mathcal{C}$ forms a closed contour in the domain of analyticity 
	$[-L,L]\times [-i,i]$ of $a$, on which the 
	phase $\phi(y):= -y^2-2iy$ has a single quadratically nondegenerate critical
	point at $y=-i$ and, for $\eps$ sufficiently small, satisfies $\Re \phi(y)\leq \Re \phi(-i)= -\frac{1}{h}$,
	with $\Re \phi(a), \Re \phi(b)<\Re \phi(-i)$.

	The estimate for $x\leq L<1$ follows, similarly, 
	defining the contour $\mathcal{C}'= [x,-i|x|+\eps]\cup [-i|x|+\eps, -i|x|-\eps]\cup [-i|x|-\eps,-x]$, 
	and observing that the phase $\phi$ has no critical points on $\mathcal{C'}$
	with $\Re \phi(y)\leq e^{-x^2/h}$ throughout.
	Expressing $\int_{-x}^x=\int_{\mathcal{C}'}$ using Cauchy's Theorem, we obtain by a nonstationary phase
	computation, integrating by parts, finally,
		\ba\label{nonstat}
		\int_{\mathcal{C}'}e^{-\phi(z)/h}a(z)dz &= -h \int_{\mathcal{C}'}(e^{-\phi(z)/h})'a(z)/\phi'(z) dz\\
&= h e^{-\phi(z)/h}a(z)/\phi'(z)|^x_{-x} + h \int_{\mathcal{C}'}e^{-\phi(z)/h}(a(z)/\phi'(z))' dz
		\lesssim he^{-x^2/h}.
		\ea

		For $x\leq 1=L$, using the original contour $\mathcal{C}$, a nonstationary phase estimate like \eqref{nonstat} shows that the contribution from the parts $[x,-i+\eps]$ and $[-i-\eps,-x]$ is $\lesssim he^{-x^2/h}$, while
		the contribution from $[-i+\eps, -i-\eps]$ may be estimated directly from \eqref{e:statest}
		as $\sim (h^{1/2} a(-i) +O(h))e^{-1/h}$, whence, summing, we obtain 
	$\int_{-x}^x e^{-\phi(z)/h}a(z)dz \sim h^{1/2} a(-i)e^{-1/h} +O(he^{-x^2/h})$ as claimed.
\end{proof}

\subsection{A $C^\infty$ oscillatory integral}\label{s:osc}
We next investigate the integral
$$
\int_{-x}^x e^{-\frac{y^2}{h}-2i\frac{y}{h}} a(y)dy
=
\int_{0}^x e^{-\frac{y^2}{h}-2i\frac{y}{h}} a(y)dy
$$
with symbol $a(y)\equiv 0$ for $y\leq 0$, $a(y)=e^{-y^{-\theta}}$ for $y>0$
that is $C^\infty$ but not analytic.

\begin{proof}[Proof of Lemma \ref{halfprop}]
	Take without loss of generality $x=+\infty$, noting that the resulting error is negligible,
	denoting the resulting quantity 
	$I(h):= \int_{0}^{+\infty} e^{-\frac{y^2}{h}-2i\frac{y}{h}} a(y)dy $.
	Define
	\be\label{ab}
	\alpha= 1-1/s \in (0,1),
	\quad
	\beta= e^{-\frac{i\pi (1-1/s)}{2} }.
	\ee

	\medskip

	{\it Case (i)} ($1<s< 2$).
	Deforming the contour $[0,+\infty]$ using Cauchy's Theorem to the 
	contour 
	$$
	z=h^\alpha \beta
	t, \quad t\in (0,+\infty),
	$$
	we obtain
	\be\label{prelim}
	I(h)\sim h^\alpha \int_0^\infty e^{\frac{i\beta( -2t -t^{-\theta} +i\beta 
	h^\alpha  t^2)}{h^{1-\alpha}}} dt , 
	\ee
	where the real part of the phase, 
	$
	\Re i \beta( -2t -t^{-\theta} +i  \beta h^\alpha t^2),
	$
	has a unique, quadratically degenerate maximum for $h=0$ at $t_0=  
2^{-(1-1/s)}$.
	Applying the Analytic Complex Stationary Phase Lemma,
	Lemma \ref{l:stat}- more precisely the generalization of
	Remark \ref{r:stat}(3) to an $h$-dependent phase,\footnote{
	With also $\mathcal{C}=[b,a]$, in the notation of the lemma.} 
	we obtain finally
	$ I(h)\sim h^{\alpha}h^{\frac{1-\alpha}{2}} e^{\Re (i\beta) 
		\big(\frac{-2t_0 - t_0^{-\theta} + i\beta h^\alpha+ O(h^{2\alpha})}{h^{1-\alpha}}\big)}, $
	yielding the result.
\medskip

	{\it Case (ii)} ($s\geq 2$). 
	In this case, $\Re (-\beta^2) = - \cos( \pi (1-1/s)) >0$, 
	and so we cannot move the contour as in the
	previous case to a ray in direction $\beta$; indeed, the integral in \eqref{prelim} is not convergent.
	Instead, we first move to the furthest possible ray
	$ z=h^\alpha e^{-i\pi_4} t$, $t\in (0,+\infty)$
	on which the infinite integral converges (and for which the contribution at infinity vanishes for all rays between, justifying the shift of contour),
	obtaining
	$
	I(h)\sim h^\alpha \int_0^\infty e^{\frac{i\beta_*( -2t -t^{-\theta} +i\beta_* 
	h^\alpha  t^2)}{h^{1-\alpha}}} dt $
	for $\beta_*= e^{-i\pi_4}$, then truncate to
	\be\label{sready}	
	I(h)\sim h^\alpha \int_0^L e^{\frac{i\beta_*( -2t -t^{-\theta} +i\beta_* 
	h^\alpha  t^2)}{h^{1-\alpha}}} dt 
	\ee
	for $L\gg1$ at the cost of a negligible $O(e^{cL/h^{1/s}})$ contribution.
	Finally, we estimate \eqref{sready} using the $h$-dependent stationary phase lemma, with countour 
	$
	\mathcal{C}= [\beta_* L, \beta L] \cup [\beta L, 0],
	$
	$\beta $ as in \eqref{ab}, yielding the same estimate 
	$ I(h)\sim h^{\alpha}h^{\frac{1-\alpha}{2}} e^{\Re (i\beta) 
		\big(\frac{-2t_0 - t_0^{-\theta} + i\beta h^\alpha+ O(h^{2\alpha})}{h^{1-\alpha}}\big)} $
	as in the previous case. 
	\end{proof}

\br
The rate of exponential decay obtained in Lemma \ref{halfprop}
may be recognized as $e^{\Re \Psi(z_*(h),h)}$, where $z_*(h)$
is a critical point of the augmented phase $\Psi(y,h):=(-y^2/h -2iy/h-1/y^\theta)$ obtained by including
the symbol $e^{-1/y^\theta}$ as part of the phase.
\er

\subsection{Gevrey-regularity complex stationary phase}\label{lebeaurmk1}
\begin{proof}[Proof of Proposition \ref{p:gevrey}]
	First observe that, by estimate \eqref{element}, we may construct cutoff functions $\xi$ of Gevrey $(s,T)$ class for arbitrary
	$s$ and (by rescaling) $T$, hence, muliplying $a$ by such a $\xi$ supported near the origin and periodically extending, the observing that the error incurred is negligible, we may restrict to the case of $a$ periodic.
	For periodic functions Foias and Temam have derived an equivalent version
$
\| a\|^*_{s,T}:= \sqrt{
\sum_{j\in \mathcal{Z}} (1+|j|)^2 e^{2T|j|^{1/s}} |\hat a_j|^2}
$
of the Gevrey $(s,T)$-norm \cite{FT}, \cite[(2.4), p. 4]{PV}, where $\hat a_j$ denotes Fourier transform, with
$T$ in the analytic case $s=1$ corresponding to width of the strip of
analyticity about the real axis $\bbR$.
For $T> 1$, and $\|a\|^*_{s,T}$ finite, the complex stationary phase integral
$ L_a(h):=\int_{\bbR} e^{-y^2/h- 2iy/h} a(y) dy$
may be seen to satisfy $ L_a(h)\lesssim C\|a\|^*_{T,s} h^{k_s} e^{-c/h^{1/s}}$, $c,C>0$,
by triangulation with the function $\alpha$, analytic on a
strip of width $T>1$, obtained by Fourier truncation of $a$ at modes
$|\xi|\leq 1/Th$, using the analytic stationary phase estimate 
$$
L_\alpha(h)\lesssim h^{1/2}\|\alpha\|^*_{T,1}e^{-1/h} \lesssim h^{1/2}|a\|^*_{T_0,s}e^{- 1/h^{1/s}}
$$
following from \eqref{e:quadtrans}, 
the fact that $\sup |\alpha|$ on the strip $[-L,L]\times [-i,i]$
of width $1$ is bounded by $\|\alpha\|^*_{T,1}$,
and the straightforward estimate 
$ \|\alpha\|^*_{T,1}\leq \|\alpha\|^*_{T,s}e^{1/h- T_0(T/h)^{1/s}} $,
together with the truncation error bound 
$$
\begin{aligned}
\sup_{x\in [-L,L]\subset \R} |\alpha -a| \lesssim 
\sum_{|j|\ge T/h} |\hat a_j|
&\leq 
 \sqrt{ 
\sum_{|j|\geq T/h} (1+|j|)^2|\hat a_j|^2
 }
 \sqrt{ \sum_{|j|\geq T/h} (1+|j|)^{-2} }\\
 &\lesssim	\| a\|^*_{s,T}
 \sqrt{ (1+|T/h|)^{-1}} e^{-T|T/h|^{1/s}}
 \lesssim h^{1/2} \| a\|^*_{s,T} e^{-c(T_0,T,s) /h^{1/s}}, 
 \end{aligned}
$$
where we have used in the first inequality Hausdorff--Young's inequality,
in the second Cauchy-Schwarz' inequality, 
and in the third the definition of $\|\cdot\|_{s,T}^*$.
The latter bound gives evidently
$$
|L_\alpha(h)-L_a(h)|\lesssim \sup_{x\in \R} |\alpha -a| 
 \lesssim h^{1/2} \| a\|^*_{s,T} e^{-c(T_0,T,s) /h^{1/s}}, 
$$
from which we obtain finally that, for some $c>0$,
$$
|L_a(h)|\leq
|L_\alpha(h)|+ |L_\alpha(h)-L_a(h)|\lesssim h^{1/2} \|a\|_{s,T_0}^* e^{-c /h^{1/s}}.
$$
\end{proof}

\br\label{ftruncrmk}
The Fourier truncation argument above gives a general way of
interpolating between results for analytic and $C^r$ functions.
For example applied to the analytic interpolation bound 
interpolation error $\leq C\|f\|_{1,T} (T+ \sqrt{1+T^2})^{-N}$
of \cite{DY} (adapted from \cite{T}), with cutoff $|\xi|\geq N$,
this yields
the more general bound 
$
\hbox{\rm error }\; \leq C\|f\|_{s,T} (T+ \sqrt{1+T^2})^{-N^{1/s}}
$
for Chebyshev interpolation of periodic functions $f$ on $[-1,1]$
with $N$ mesh points.
\er

\br[$C^r$ stationary phase]\label{crrmk}
Similarly, the $W^{r,\infty}$ symbol $a(y)=0$ for $y\leq 0$, $a(y)=y^r$ for $y>0$ yields
by a standard nonstationary-phase argument involving $r$ integrations by parts
$$
\int_{-x}^x e^{-y^2/h-2iy/h} a(y) dy = \int_{0}^x e^{-y^2/h-2iy/h} a(y) dy \sim h^{r}
\quad \hbox{\rm for $0<c_0\leq x$},
$$
verifying sharpness of the nonstationary van der Korput bound for $W^{r,\infty}$ symbols $a$ \cite{M}.
\er


\section{Counterexamples}\label{s:cegs}
We complete our analysis by carrying out the 
counterexamples described in the introduction.

\subsection{Block-diagonalization and decay of oscillatory integrals}\label{s:tri}
Recall the triangular system 
$$
hW'=AW:=\bp \lambda_1& h^p \theta\\0 & \lambda_2\ep W
$$
introduced in \eqref{ntrisys}
of the introduction,
with $\lambda_1(x)= x+i$, $\lambda_2=-(x+i)$, $W\in \C^2$, $p\geq 1$.

\begin{proof}[Proof of Lemma \ref{l:oscequiv}]
	We seek a coordinate change $W=TZ$ with $T$, $T^{-1}$ uniformly bounded
	in $C^1$, converting \eqref{ntrisys} to an
	exactly diagonal system $hZ'=DZ$.
	From $p\geq 1$, uniform boundedness in $C^1$, and the relation
	$D= T^{-1}AT - hT^{-1}T'$, we find, first, that 
	$
	D=\bp \lambda_1 + O(h)& 0\\0 & \lambda_2+ O(h)\ep.
	$
	Comparing $O(1)$ terms in the off-diagonal entries of
	$T^{-1}AT$, we have also 
	that $T_{11}T_{12}, T_{21}T_{22} \lesssim h$,
	and, since $\det T$ must be bounded below by boundedness
	of $T^{-1}$, $T_{11}T_{22}-T_{12}T_{21}\gtrsim 1$.

	The latter observations imply that either 
	(i) $T_{11}, T_{22}\sim 1$ and $T_{12}, T_{21} \lesssim h$, or
	(ii) $T_{12}, T_{21}\sim 1$ and $T_{11}, T_{22} \lesssim h$.
	Noting that case (ii) may be reduced to case (i) by exchanging $T$ for
	$
	\tilde T=T\bp 0 & 1\\ 1 & 0\ep,
	$
	we may assume without loss of generality (i).
	By a further, diagonal transformation rescaling the diagonal entries of $T$
	to one, we may arrange that $T=\bp 1 & ha\\ h\hat a& 1\ep$,
	with $(a, \hat a)(x,h) \lesssim 1$.
	Comparing off-diagonal coefficients of $T^{-1} AT- hT^{-1}T'$, we find by an
		induction argument, finally, that 
\be\label{Tform}
		T=\bp 1 & h^p\alpha\\ h^p\hat \alpha& 1\ep,
\quad
	D=\bp \lambda_1 + O(h^p)& 0\\0 & \lambda_2+ O(h^p)\ep;
\qquad	
	(\alpha, \hat \alpha)(x,h) \lesssim 1.
\ee
(Alternatively, for purposes of this argument, we could take without loss of
generality $p=1$.)

This reduces us to the case treated in Theorem \ref{wasow1}, with $\Theta_{11}=
\Theta_{22}=\Theta_{21}\equiv 0$, $\Theta_{12}= \theta$, $\alpha_{12}=\alpha$
and $\alpha_{21}=\hat \alpha$, in which \eqref{ricatti} specializes to
		\ba\label{ricatti2}
		h \alpha'&= (\lambda_1-\lambda_2)\alpha +  \theta ,\\
		h \hat \alpha'&= (\lambda_2-\lambda_1)\hat \alpha 
		- h^p \hat \alpha \theta \alpha ,\\
\ea
with $d_1= \lambda_{1} + h^p \theta \hat \alpha $ and $ d_{2} = \lambda_2$.
Noting that \eqref{ricatti2}(i) is independent of \eqref{ricatti2}(ii), and that
\eqref{ricatti2}(ii) is consistent with $\tilde \alpha\equiv 0$, we find that
{\it there exists a diagonalizing transformation if and only if there exists a 
triangular one, $\hat \alpha\equiv 0$, which exists if and only if 
there is a solution $\alpha$ of \eqref{ricatti}(i) that is uniformly bounded as $h\to 0^+$,}
equivalent by Duhamel's principle to uniform boundedness of
\ba\label{t12}
\alpha(x,h) &=
h^{-1}\int_0^x e^{h^{-1}\int_y^x (\lambda_1 -\lambda_2)(z)dz} \theta(y) dy
- e^{h^{-1}\int_0^x (\lambda_1 -\lambda_2)(z)dz} \alpha(0,h) \\
&=
e^{(x^2+2ix)/h}\Big(h^{-1}\int_0^x e^{ -(y^2+2iy)/h} \theta(y) dy - \alpha(0,h)\Big).
\ea

\medskip

{\it ($\Leftarrow$):}
Noting that 
$ e^{-2ix}\alpha(x) - e^{2ix}\alpha(-x) =
e^{x^2/h}\int_{-x}^x e^{ -(y^2+2iy)/h} \theta(y) dy, $
we see that \eqref{t12} is uniformly bounded as $h\to 0^+$
only if $ \int_{-x}^x e^{ -(y^2+2iy)/h} \theta(y) dy \lesssim h e^{x^2/h}$.

\medskip

{\it ($\Rightarrow$):} 
Choosing $\alpha(0,h)= h^{-1}\int_0^L e^{ -(y^2+2iy)/h} \theta(y) dy $,
we obtain from \eqref{t12}
\be\label{1}
\alpha(x,h)=
e^{(x^2+2ix)/h}h^{-1}\int_x^L e^{ -(y^2+2iy)/h} \theta(y) dy,
\ee
or, alternatively,
\ba\label{2}
\alpha(x,h)&=
e^{(x^2+2ix)/h}h^{-1}\int_x^{-L} e^{ -(y^2+2iy)/h} \theta(y) dy\\
&\quad +
e^{(x^2+2ix)/h}h^{-1}\int_{-L}^{L} e^{ -(y^2+2iy)/h} \theta(y) dy.
\ea
For $0\leq x\leq L$, a nonstationary phase estimate as in the proof of Lemma \ref{quadstat} (case $|x|\leq 1$)
yields uniform boundedness of the righthand side of \eqref{1} as $h\to 0^+$; 
for $-L\leq x\leq 0$, the same estimate yields uniform boundedness of 
the $\int_x^{-L}$ term on the righthand side of \eqref{2}, while assumption \eqref{e:onlyif} yields uniform boundedness
of the $\int_{-L}^L$ term.  Thus, $\alpha$ is uniformly bounded for all $|x|\leq L$.

\end{proof}

\br\label{nonzerormk}
From the case $\lambda_j$ constant, we see that
in general $\alpha(0,h)\ne 0$, or $T(0,h)\neq \Id$.
\er

\subsection{Global counterexample}\label{s:global}

\begin{proof}[Proof of Corollary \ref{c:global}]
	For any nonvanishing $\theta$ analytic on $[-1,1]\times [-i,i]$ and continuous on
	$[-L,L]$, we find by Lemma \ref{quadstat} that 
	condition \ref{e:onlyif} is violated for $L>1$.	
\end{proof}

\subsection{$C^\infty$ counterexample}\label{s:cinfty}
\begin{proof}[Proof of Corollary \ref{c:global}]
	For the Gevrey class $s$ function defined by $\theta(y)\equiv 0$ for $y\leq 0$ and
	$\theta(y)= e^{-y^{-\theta}}$ for $y>0$, $\theta=\frac{1}{s-1}$, 
		we find by Lemma \ref{halfprop} that 
	condition \ref{e:onlyif} is violated for any $L>0$.	
\end{proof}

\br[$C^r$ counterexample]\label{crceg}
For the $W^{r,\infty}$ function defined by $\theta(y)\equiv 0$ for $y\leq 0$ and $\theta(y)=y^r$ for $y\geq 0$, we find by
Remark \ref{crrmk} that condition \ref{e:onlyif} is violated for any $L>0$, giving a particularly elementary counterexample
to existence of diagonalizers in the $C^{r-1}$ case, $r$ arbitrary.
\er

\subsection{Regular-singular point counterexample}\label{s:rs}
Finally, for singular ODE
$
w'= (1/z)A(z,h)w,
$
with $A$ analytic on the disk $B(0,r)$, we show that there {\it may not}
always exist an analytic block-diagonalizing conjugator on the whole disk, but only on a slit disk as described in Corollary \ref{singprop2}.
%

\begin{example}\label{sliteg}
Consider 
$
hzW'= \bp 1 & 0 \\ 0 & 0\ep W + h \bp 0 & z \phi\\0 & 0\ep W,
$
$\phi$ analytic.
Looking without loss of generality 
(following Section \ref{s:tri})
for a triangular conjugator  
$T=\bp 1 & h\alpha  \\ 0 & 1\ep$,
we obtain the ODE
$
h\alpha'= \alpha/z + \phi.
$
Assuming Taylor expansions $\alpha(z;h)=\sum_j \alpha_j(h) z^j$, $\phi(z)=\sum_j \phi_j z^j$, we find, comparing
coefficients of like order, that $\alpha_0=0$ and $(jh-1)\alpha_j= \phi_{j-1}$ for $j\geq 1$.
Thus, there is a formal solution for all $0<h\leq h_0$ if and only if $\phi$ is polynomial, in which case there is a polynomial solution $\alpha$.  For, otherwise, taking $h_j=1/j$, we have a sequence $h_j\to 0^+$ for which there is no solution whenever
$\phi_{j-1}\neq 0$.
Rewriting the ODE as $(z^{-\frac{1}{h}}\alpha)'=hz^{-\frac{1}{h}}\sum_j\phi_j z^j$ and integrating term by term, we see explicitly the appearance of a $\phi_{j-1}z^j \log z$ term when $h=1/j$ and $\phi_{j-1}\neq 0$.
Thus, in general there may not exist an analytic diagonalizer for the regular singular point problem \eqref{sing}.  
On the other hand, if blocks $A_{11}(0,0)$ and $A_{22}(0,0)$ have no eigenvalues differing by a real value.
one can show that an analytic solution $(\alpha_1,\alpha_2)$ exists for $|z|\leq r$, $0<h\leq h_0$, by
solving for the formal power series as above, and showing convergence by direct estimates of coefficients.
\end{example}



\appendix



\section{Analyticity of traveling-wave profiles}\label{s:profiles}


\subsection{An Analytic Stable Manifold Theorem}

Consider a complex analytic ODE 
\be\label{nonline}
 (d/dt)u=f(u),\qquad t\in \CC^1,\, u\in \CC^n, \,,  f\in C^\omega:\CC^n\to \CC^n,
\ee
defined in a neighborhood of an equilibrium $u_*$,
$f(u_*)=0$, with associated linearized equation 
\be\label{line}
v'=Av,
\qquad A:=df(u_*).
\ee
Associated with $A$, define the stable subspace $\Sigma_s$
as the direct sum of all eigenspaces of $A$ associated to
stable eigenvalues, or eigenvalues with strictly negative real
part.  Likewise, define the center, and unstable
subspaces $\Sigma_c$ and $\Sigma_u$ as the direct
sum of eigenspaces associated to neutral and unstable eigenvalues,
respectively, i.e., eigenvalues with zero and strictly positive
real part.  This gives a decomposition 
$\CC^n= \Sigma_s \oplus
\Sigma_c \oplus \Sigma_u$
of $\CC^n$ into subspaces that are invariant under the flow of \eqref{line}.

Defining associated (total) eigenprojections  $\Pi_s$,
$\Pi_c$, and $\Pi_u$ as the sum of all eigenprojections 
associated with stable, neutral, and unstable eigenvalues, respectively,
we have the standard bounds
\ba\label{expbds}
|e^{At}\Pi_s|&\le C(\eta)e^{-\eta |\Re t|}, \quad \Re t\ge 0,
\; |\Im t|\leq \nu |\Re t|,\\
|e^{At}\Pi_c|&\le C(\eta)e^{\theta |t|} \qquad  |\Im t|\leq \nu |\Re t|,\\
|e^{At}\Pi_u|&\le C(\theta) e^{-\eta |\Re t|}, \quad \Re t\le 0,
\; |\Im t|\leq \nu |\Re t|,\\
\ea
for any $\eta>0$ strictly smaller than the minimum absolute value
of the real parts of stable and unstable eigenvalues,
$\theta>0$ arbitrarily small, and some $\nu>0$.

\bt [Analytic Stable Manifold Theorem] \label{t:stableman}
For $f\in C^\omega$, there exists local to $u_*$
a $C^\omega$ stable manifold $\cM_s$
tangent at $u_*$ to $\Sigma_s$, expressible
in $w:=u-u_*$ as a $C^\omega$ graph 
$\Phi_s:\Sigma_s\to \Sigma_c\oplus \Sigma_u$,
that is (locally) invariant under the flow of \eqref{nonline}
and uniquely determined by the property that solutions in
$\cM_s$ decay exponentially to $u_*$ in ``approximately forward'' time, 
in the sense that
$|w(t)|\le C(\tilde \eta)e^{-\tilde \eta|\Re t|}|w(0)|$
for $\Re t\geq 0$, $|\Im t|\leq \nu |\Re t|$,
and any $0<\tilde \eta< \eta$, $\eta, \nu$ as in \eqref{expbds}.
\et

\begin{proof}
Defining $w:=u-u_*$, we obtain the nonlinear perturbation equation
$
w'=Aw + N(w),
$
where $A:=(df/du)(u_*)$ is constant and 
$
N( w):= f(u_* + w)- f(u_*)- (df/du)(u_*))w
$
is the first-order Taylor remainder for $f(u)-f(u_*)$, 
satisfying $N(0)=0$, $(dN/d w)(0) =0$, hence
\be\label{e:Nbds}
\lim_{|w|\to 0^+}|N(w)|/|w|=0,
\qquad
\lim_{|w|\to 0^+}|(dN/dw)(w)|=0.
\ee

Applying projections $\Pi_j$, $j=s,u,c$, and using the fact that
these commute with $A$, we may coordinatize as
three coupled equations 
$
(\Pi_j w)'=A(\Pi_j w) + \Pi_j N(w)
$
in variables $w_s:=\Pi_s w$,
$w_c:=\Pi_c w$, and $w_u:=\Pi_u w$.
Considering $\Pi_j N$ as inhomogeneous source terms and applying
the variation of constants formula, 
we obtain
\be\label{vareq}
\Pi_j w(t)=
e^{A(t-t_{0,j})}\Pi_j w(t_{0,j}) + \oint_{t_{0,j}}^t e^{A(t-s)}\Pi_j  N(s, w(s))\,ds,
\ee
$j=s,c,u$, so long as the solution $w$ exists.

Under assumption $|w(t)|\le Ce^{-\tilde \eta \Re t}$
for $\Re t\geq 0$, $|\Im t|\leq \nu |\Re t|$,
we find by \eqref{expbds} that $|e^{A(t-\tau)}\Pi_j w(\tau)|$
decays exponentially as $\Re \tau\to +\infty$ for $j=c,u$
with $|\Im (\tau-t)|\leq \nu |\Re (\tau-t)|$
and  $\Re t\geq 0$, $|\Im t|\leq \nu |\Re t|$,
while $ \oint_{t}^{\tau}e^{A(t-s)}\Pi_j  N(s, w(s))\,ds$ 
(since $|N(w,s)|\le C|w(s)|$ for $|w|$ uniformly bounded) converges
to a limit.
Thus, taking $t_{0,j}\to +\infty$ for $j=c,u$, in \eqref{vareq},
we find that the first (linear) term on the righthand side disappears,
leaving
$\Pi_j w(t)= -\oint_{t}^{+\infty}e^{A(t-s)}\Pi_j  N(s, w(s))\,ds$.
Choosing $t_{0,s}=0$ and summing the three equations \eqref{vareq},
we obtain for $w_s:=\Pi_s w(0)$ the integral fixed-point representation
\ba\label{Trep}
w(t)=T(w_s, w)&:=
e^{At}w_s + \oint_{0}^t e^{A(t-s)}\Pi_s  N( w(s))\,ds
-\oint_{t}^{+\infty}e^{A(t-s)}\Pi_{cu}  N( w(s))\,ds,
\ea
valid for all solutions decaying exponentially in approximately forward time
at rate $Ce^{-\tilde \eta \Re t}$, where $\Pi_{cu}:=\Pi_c + \Pi_u$
denotes the total eigenprojection of $A$ onto the center--unstable
subspace $\Sigma_{cu}:= \Sigma_c\oplus \Sigma_u$.
Here, $\oint_t^{+\infty}dt$ denotes {\it any} contour integral lying
within the cone $\Re (s-t)\geq 0$, $|\Im (s-t)|\leq \nu |\Im (s-t)|$
with $\Re s\to +\infty$.

Define the time-weighted norm 
$
\|f\|_{\tilde \eta}:=\sup_{\Re t\ge 0, \,
|\Im t|\leq \nu|\Re t|}e^{\tilde \eta \Re t}|f(t)|,
$
noting $ |f(t)|\le e^{-\tilde \eta \Re t} \|f\|_{\tilde \eta}$ for 
$\Re t\ge 0$. $|\Im t|\leq \nu|\Re t|$.
With respect to this norm, we find by a straightforward (and standard, in
the real setting)
computation using
\eqref{expbds}
that, for $|w_s|$ and $\delta>0$ sufficiently small, 
the integral operator $T$ is Lipschitz in $w_s$ and
contractive in $w$ from the $\|\cdot\|_{\tilde \eta}$-ball 
$B(0,\delta)$ to itself:
\be\label{e:contTs}
\|T(w_s,w_1)-T(w_s,w_2)\|_{\tilde \eta}\le (1/2)|w_1-w_2\|_{\tilde \eta}.
\ee
Thus, by the Contraction Mapping Principle
there exists a unique solution $w=w(w_s)\in B(0,\delta)$ in 
$\cB_{\tilde \eta}:=\{f:\, \|f\|_{\tilde \eta}<\infty \}$. 
Appealing, finally, to the analytic dependence
of analytic fixed point mappings on parameters, we obtain that 
$w(w_s)$ is a $C^\omega$ function from 
$\Sigma_s\subset \RR^n\to \cB_{\tilde \eta}$.

Evidently (computing $\Pi_{s}w(0)=w_s$ by applying $\Pi_{s}$ to
the right-hand side of \eqref{e:contTs} and using $\Pi_s\Pi_{cu}=0$), 
the function $w(w_s)$ takes each point in $\Sigma_s$ to 
a unique exponentially decaying solution on $\Re t\ge 0$,
$|\Im t|\leq \nu |\Re t|$
with initial data satisfying $\Pi_s w(0)=w_s$.
Thus, defining 
$
\Phi(w_s):=\Pi_{cu}w(w_s)|_{t=0},
$
we obtain the desired $C^\omega$ map, whose graph, a $C^\omega$ manifold, consists, 
locally, of the set of
data with exponentially decaying solutions, or, equivalently,
the invariant set of orbits decaying exponentially in approximately
forward time.
Tangency, $\Phi(0)=0$ and $d\Phi(0)=0$, then follow
by uniqueness together with the fact that
$N(0)=0$, $(dN/d w)(0) =0$.
\end{proof}

\br\label{profrmk}
The above argument can be recognized as the standard fixed-point construction for ODE on the real line, together with Cauchy's Theorem plus the observation that the standard linearized bounds \eqref{expbds} hold on a wedge.
In the absence of a center subspace, essentially the same argument
yields that the stable manifold is uniquely decribed as the set of
solutions merely bounded in approximately forward time.
Indeed, setting $\nu=0$, we find that all solutions merely 
bounded and close to $u_*$
in forward time for $t$ real must lie in the stable manifold as well.
\er

\subsection{Application: analyticity of profiles}\label{s:appman}
Applying the above stable (unstable) manifold theorem to a profile ODE,
we obtain immediately a result of {\it analyticity on an approximately
forward (approximately backward) sector}.
Consider a general profile ODE
\be\label{prof}
u'=f(u),
\ee
written as a first-order system, with $f$ analytic on the
set $u\in \mathcal{U}\subset \CC^n$ in question and $'$ denoting $(d/dx)$.
Here, the connection problem can be either a heteroclinic/homoclinic
one on the whole real line, or a boundary-value problem on a forward
or backward real half-line.
In any case, provided the desired endstate in forward real time is
a hyperbolic rest point for the ODE restricted to the real axis,
we find from Remark \ref{profrmk} that any connecting profile 
(since bounded and close to $u_*$, at the least) must belong to the
complex analytic stable manifold described above, consisting entirely
of orbits analytic and exponentially decaying
in an approximately forward sector in $x$.
More generally, this holds for any exponentially decaying profile,
even in the presence of a center subspace at $u_*$.

\section{Complex saddlepoint estimate/method of stationary phase}\label{s:stat}

\bl[Analytic Stationary Phase Lemma: adapted from \cite{S,PW}]\label{l:stat}
Let $\mathcal{C}\cup [a,b]\subset \C$ be an oriented continuous closed curve in the complex plane, $a$, $b$ real, with standard orientation on $[a,b]$, enclosing domain $\Omega$,
with $\phi$, $A$ analytic in $x$ on $\bar \Omega$ and $A$ Lipschitz continuous in $h$ for $0\leq h\leq h_0$, $h_0>0$. Suppose further that there is a quadratically nondegenerate critical point $z_0\in \mathcal{C}^{int}$ 
of $\phi$ such that $\Re \phi$ has a nonstrict maximum on $\mathcal{C}$ at $z_0$, 
with $\mathcal{C}$ described by
a $C^2$ curve $\gamma(t)$, $t\in [-\eps,\eps]$ in the vicinity of $z_0$, $\gamma(0)=z_0$, and this is the only critical
point on $\bar {\mathcal{C}}$.  Then, 
\be\label{e:statest}
I(h):= \int_a^b e^{\frac{\phi(x)}{h}} A(x,h)dx= 
h^{1/2} e^{\frac{\phi(z_0)}{h}} \big(  A(z_0,0) \sqrt{ \frac{2\pi}{-\phi''(z_0)}} + O(h^{-1/2}) \big) 
\ee
as $h\to 0^+$, where the square root is chosen so that $-\gamma'(0)$ and $\frac{1}{ \sqrt{-\phi''(z_0)}}$
lie in the same half-plane.
\el

\begin{proof}
	By the Cauchy integral formula,
	$  \oint_{ \mathcal{C}\cup [a,b] } e^{\frac{\phi(z)}{h}} A(z,h)dz=0$,
	whence, rearranging, 
	$$
	I(h) e^{-\frac{\phi(z_0)}{h}} 
	=
	-\int_{ \mathcal{C} } e^{\frac{\phi(z)-\phi(z_0)}{h}} A(z,h)dz,
	$$
	where, by assumption, $\Big| e^{\frac{\phi(z)-\phi(z_0)}{h}} \Big|$ is uniformly bounded.
	Substituting $A(z,0)$ for $A(z,h)$ gives a negligible $O(h)$ error, reducing to the classical case treated
	in \cite{S,PW}.
	Next, loosely following \cite{PW}, 
	deform $\gamma(t)$ to $\tilde \gamma(t):= \gamma(t) - \eta \chi(t)\nabla (\Re \phi)$, where $\chi(t)$
	is a $C^\infty$ cutoff function vanishing at $t=\pm \eps$ and identically $1$ in a neighborhood of $z_0$,
	and $\eta>0$ is taken sufficiently small that the image of $\tilde \gamma$, and the domain between the images
	of $\gamma$ and $\tilde \gamma$, lies in the region of analyticity of $A(\cdot, 0)$, $\phi$: that is, an approximate
	gradient descent.
	It is readily seen that, for $\eta$ sufficiently small, 
	$\tilde \gamma(0)=z_0$ and $\Re \phi(\tilde \gamma)$ has a quadratically nondegenerate maximum at $t=0$,
	while $\Re (\phi(\tilde \gamma(t))-\phi(z_0)\leq 0$ for $t\in [-\eps,\eps]$.
	Now, for $|t|\geq h^{5/12}$, we may write the integral over $\tilde \gamma(t)$ as
	$$
	\int e^{\frac{\phi(z)-\phi(z_0)}{h}} A(z,0)=\int  h \partial_z( e^{\frac{\phi(z)-\phi(z_0)}{h}}) A(z,0)/\phi'(z)
	$$
	and obtain a uniform $O(h)$ bound by integrating by parts.
%
For $|t|\leq h^{5/12}$, meanwhile, we have $t^3/h<<1$, so that $e^{O(t^3/h)}= 1+ O(t^3/h)$.
Taylor expanding, therefore, and keeping only lower order terms,
	we find that on this region
	$-\int  e^{\frac{\phi(z)-\phi(z_0)}{h}}A(z,0)dz$ reduces modulo $O(h)$ errors to the exact integral
	$
	\int_\R e^{\frac{\phi''(0)(\tilde \gamma'(0)t)^2}{2h}}A(z_0,0) (-\tilde \gamma'(0)) dt
	=
	h^{1/2} A(z_0,0) \sqrt{ \frac{2\pi}{-\phi''(z_0)}},
	$
	giving the result.  Here, we are bounding remainder terms using 
	$
	\int_\R e^{-\frac{\eta t^2}{h}}(|t|+ |t^3|/h)dt\leq h^{-1/2}\int e^{-\frac{\eta t^2}{2h}}dt
	=O(h^{-1})
	$
	and $ \int_{|t|\geq h^{5/12}} e^{-\frac{\eta t^2}{2h}}dt = O(h^{1/2}e^{-\frac{\eta h^{10/12}}{h}})\ll h^{-1},  $
	and using the change of coordinates $z-z_0= 
	w \sqrt{ \frac {|\phi''(z_0)|} {-\phi''(z_0)}}$ 
	to reduce the exact integral to a standard Gaussian integral 
	$$
	A(z_0,0) \sqrt{ \frac {|\phi''(z_0)|} {-\phi''(z_0)} }
		\int_\R e^{-\frac{|\phi''(0)|w^2}{2h} }dw
	=  A(z_0,0) \sqrt{ \frac{2\pi h}{-\phi''(z_0)}} \;.
	$$
\end{proof}

\brs\label{r:stat}
1. Lemma \ref{l:stat} may be generalized in straightforward manner 
	to the case that $\mathcal{C}$ contains multiple critical points, or critical points at its endpoints \cite{PW},
	yielding an asymptotic expansion as the sum of contributions of each critical point (with endpoints counting half).
	Likewise, for $A(z_0)=0$, and making the additional assumption of strict inequality
	$\Re \phi(a),\Re \phi(b)<\Re \phi(z_0)$, so that boundary terms at $a,b$ are neglible, 
	a standard Taylor expansion/inverse Fourier transform computation argument near 
	$z_0$ \cite{M}[p. 60, exercise (ii]) 
	combined with repeated integration by parts away from $z_0$ yields 
	$$
	I(h) \sim h^{(j+1)/2} e^{\frac{\phi(z_0)}{h}}
	$$
	 in place of \eqref{e:statest}, where $j$ is the order of the first nonvanishing derivative of $\phi$ at $z_0$.

	 2. Lemma \ref{l:stat} might more properly
	 be titled as the {\it Saddlepoint Lemma,} as it implicitly involves a preliminary deformation
	of $[a,b]$ to a curve $\mathcal{C}$ passing through a saddlepoint $z_0$, 
	whose contribution is then estimated by the method of stationary phase.
	Note the difference in emphasis between estimate \eqref{e:statest} and those of classical 
	stationary phase estimates focused on critical points on $[a,b]$, for example, real stationary phase or ``Fourier transform'' type estimates, or the complex-valued version of \cite{PW}.  These estimates would typically give a bound of $O(h^{-N})
e^{\min_{[a,b]}\Re \phi/h}$, $N$ arbitrary, for $I(h)$, reflecting absence of critical points in $[a,b]$, but
not capturing the full exponential rate of decay of \eqref{e:statest}.
Here, we are focused on the sharp rate of exponential decay, rather than details of the asymptotic expansion of the 
$O(h^{1/2})$ algebraic factor as in the classical case.
Put another way,
the classical phrasing of the stationary phase lemma as in \cite{PW} corresponds to the case $\mathcal{C}=[b,a]$.

3. A useful extension, allowing phases $\phi=\phi(x,\theta(h))$ depending analytically on $x$, $\theta$, 
$\theta $ monotone decreasing with $\lim_{h\to 0^+}\theta(h)=0$, is the estimate
\be\label{e:statext}
\int_a^b e^{\frac{\phi(x,\psi(h))}{h}} A(x,h)dx
=h^{1/2} e^{\frac{(\phi(z_0,\theta(h))+ O(\theta(h)^2)}{h}} \big(  A(z_0,0) \sqrt{ \frac{2\pi}{-\phi''(z_0)}} + O(h^{-1/2}+\theta(h)) \big), 
\ee
which follows by first performing the above-described approximate gradient flow for the limiting flux $\phi(\cdot, 0)$ to
reduce to the case that $z_0$ is a strict minimizer of $\Re \phi(\cdot, 0)$ on the image of $\tilde \gamma$, quadratically nondegenerate at $t=0$, then noting that this situation persists under small analytic perturbations, by the Implicit Function Theorem and continuity, with the minimizer $z_0^h $ of $\phi(\cdot, h)$ distance $O(\theta(h))$ from the minimizer for $\phi(\cdot, 0)$.
We make use of this extension in our computation of \eqref{e:quadtrans} above.
\ers


\end{document}